\documentclass[12pt]{amsart}

\usepackage{fullpage, graphicx, enumitem}
\usepackage{url}
\usepackage{verbatim}
\usepackage{color}

\usepackage{bbm}
\usepackage{amssymb,amsmath,amsthm,mathrsfs,mathtools,thmtools,latexsym}
\usepackage{amsfonts,savesym,graphicx,bm,amsthm,stmaryrd}
\usepackage[mathscr]{eucal}
\usepackage[all]{xy}
\usepackage{lipsum}
\usepackage{rotating}
\usepackage{float}
\usepackage{gensymb}
\usepackage{upgreek}
\usepackage[thinc]{esdiff}
\usepackage{tikz-cd}
\usepackage{pgfplots}
\pgfplotsset{compat=newest}
\usepackage{xcolor,textpos}
\usepackage{helvet}

\definecolor{hot}{RGB}{65,105,225}
\usepackage[pagebackref=true,colorlinks=true, linkcolor=hot ,  citecolor=hot, urlcolor=hot]{hyperref}
\usepackage{caption}
\usepackage[subrefformat=parens, labelfont=up]{subcaption}

\newtheorem{theorem}{Theorem}[section]
\newtheorem{lemma}[theorem]{Lemma}
\newtheorem{conjecture}[theorem]{Conjecture}

\newtheorem{proposition}[theorem]{Proposition}

\newtheorem{openquestion}[theorem]{Open Question}

\theoremstyle{definition}
\newtheorem{definition}[theorem]{Definition}

\theoremstyle{remark}

\newtheorem{remark}[theorem]{Remark}
\newtheorem{example}[theorem]{Example}

\numberwithin{equation}{section}

\newcommand{\Div}[1]{\operatorname{div}\left(#1\right)}

\newcommand{\Zeta}[3]{Z_{\operatorname{top},#3}(#1;#2)}
\newcommand{\Zetam}[3]{Z_{\operatorname{mot},#3}(#1;#2)}
\newcommand{\Zetaw}[4]{Z_{\operatorname{top},#3}(#1,#4;#2)}
\newcommand{\Zetamw}[4]{Z_{\operatorname{mot},#3}(#1,#4;#2)}

\newcommand{\C}{\mathbb{C}}

\newcommand{\Q}{\mathbb{Q}}
\newcommand{\Pm}{\mathbb{P}}

\newcommand{\A}{\mathbb{A}}

\newcommand{\LA}{\mathbb{L}}
\newcommand{\M}{\mathcal{M}_\C}

\newcommand{\GR}{K_0(\operatorname{Var}_\C)}

\author{Lise Fonteyne}
\address{Department of Mathematics, KU Leuven, Celestijnenlaan 200B, 3001 Leuven, Belgium.}
\email{lise.fonteyne@kuleuven.be}

\author{Willem Veys}
\address{Department of Mathematics, KU Leuven, Celestijnenlaan 200B, 3001 Leuven, Belgium.}
\email{wim.veys@kuleuven.be}

\thanks{L. Fonteyne was supported by long term
	structural funding (Methusalem grant) by the Flemish Government.  W. Veys was supported by KU Leuven Grant C16/23/010.
}

\title{On a generalized monodromy conjecture for curves using differential forms}

\begin{document}
\begin{abstract}
Motivic and topological zeta functions are singularity invariants, mainly associated to a function $f$ and a top differential form $\omega$ on a smooth variety. When $\omega$ is the standard form $dx_1\wedge \dots \wedge dx_n$ on affine $n$-space, the monodromy conjecture states that poles of these zeta functions should induce monodromy eigenvalues of $f$.
We study natural generalized statements of the monodromy conjecture for functions $f$ on complex surface germs; more precisely on singular surfaces for forms $\omega$ that generalize the standard form, and on the affine plane for forms $\omega$ that are intrinsically associated to $f$.

For all cases, we provide counterexamples to the statement. In addition, when the intrinsically associated $\omega$ is given by the generic polar of $f$, we discover a relation between the poles of the zeta functions and the intersection behaviour of the polar curve.
\end{abstract}

\maketitle

\section{Introduction}

The motivic and topological zeta functions are important singularity invariants of a polynomial function $f\in \C[x_1,\dots,x_n]\setminus \C$, more precisely of the hypersurface defined by $f$ in affine $n$-space.  Local versions are associated to the germ of $f$ at the origin $o$ of affine space, assuming $f(o)=0$.
The local motivic zeta function $\Zetam{f}{s}{o}$ of $f$ is defined as the formal generating series for classes (in the Grothendieck ring of algebraic varieties) of $m$-jets on $\A^n$ that have contact $m$ with $f$.
In fact, $\Zetam{f}{s}{o}$ turns out to be a rational function.
Using a formula for $\Zetam{f}{s}{o}$ in terms of an embedded resolution of singularities for $f^{-1}\{0\}$, it specializes to the topological zeta function $\Zeta{f}{s}{o}$, which is just a rational function in the variable $s$ with coefficients in $\Q$.
See Subsection \ref{zetafunctions} for the exact definitions.

The famous monodromy conjecture asserts that poles of $\Zetam{f}{s}{o}$ or $\Zeta{f}{s}{o}$ should induce monodromy eigenvalues of $f$, see Subsection \ref{monconj}.
It was proven by Loeser when $n=2$ \cite{loeser}, but it is still wide open for $n\geq 3$.

\medskip

The formula for the topological zeta function in the case of  plane curves is as follows.
Let $f \in \C[x,y] \setminus \C$, and assume that $f(o)=0$. Let $h:Y \rightarrow U$ be an embedded resolution of the germ at the origin of $f^{-1}\{0\}\subset \A^2$, that is,
$U$ is a small enough neighbourhood of $o$ in $\A^2$, $h$ is a proper birational morphism,  $Y$ is smooth, the restriction of $h$ to $Y \setminus h^{-1}(f^{-1}\{0\})$ is an isomorphism and $h^{-1}(f^{-1}\{0\})=\cup_{j \in T} E_j$ is a simple normal crossings divisor, i.e., the components $E_j$ are non-singular curves that intersect transversally.

	 The \emph{numerical data} of each component $E_j$ are $(N_j,\nu_j)$, where $N_j$ and $\nu_j-1$ are the multiplicities of $E_j$ in the divisor defined by $f \circ h$ and $h^*(dx \wedge dy)$, respectively. So
	$$\Div{f \circ h} = \sum_{j \in T} N_jE_j \quad\text{ and } \quad\Div{h^*(dx \wedge dy)} = \sum_{j \in T} (\nu_j-1)E_j.$$
Note that all $N_j$ and $\nu_j$ are positive integers.

We write $T = T_e \sqcup T_s$, where $T_e$ runs over the exceptional components and $T_s$ over the components of the strict transform of $f^{-1}\{0\}$, and we set $E_i^\circ = E_i\setminus \cup_{\ell\in T\setminus\{i\}}E_\ell$ for $i\in T_e$.
Then the \emph{(local) topological zeta function} of $f$  is
	\begin{equation}\label{formula-intro}
		\Zeta{f}{s}{o} 
= \sum_{j\in T_e} \frac{\chi(E_j^\circ)}{\nu_j +N_js} + \sum_{i \neq j\in T} \frac{\chi(E_i \cap E_j)}{(\nu_i +N_is)(\nu_j+N_js)},
	\end{equation}
	where $\chi(\cdot)$ denotes the topological Euler characteristic.
The statement of the monodromy conjecture is as follows. \emph{If $s_0$ is a pole of $\Zeta{f}{s}{o}$,  then $e^{2\pi i  s_0}$ is a monodromy eigenvalue of $f$ at some point of $f^{-1}\{0\}$ close to $o$.}

We recall the monodromy concept and state the general versions of the conjecture in Subsection \ref{monconj}.

\medskip
It is a natural question to look for some  generalization of the monodromy conjecture for a function (germ) $f$ on a  singular ambient variety. In dimension two, the first case to investigate is an ambient normal surface $S$. The monodromy concept is the same as for functions on smooth surfaces, but it is a priori not clear what the generalization of the zeta functions should be.

Let us look at the ingredients in formula (\ref{formula-intro}).  We can consider an embedded resolution $h:Y \rightarrow S$  of $f^{-1}\{0\}\subset S$, and use the same notations $E_j,j\in T,$ and  $\Div{f \circ h} = \sum_{j \in T} N_jE_j $.  But what about the $\nu_j$?
There are two natural generalizations.

\medskip
A first possibility is to view $\Div{h^*(dx \wedge dy)} = \sum_{j \in T} (\nu_j-1)E_j$ above as the relative canonical divisor $K_{Y|\A^2}$, and generalize it accordingly to
the relative canonical divisor $K_{Y|S}$, that is, view the $\nu_j$ as log discrepancies. Note that then these generalized $\nu_j$ are in general {\em  rational} numbers, that can be positive, negative or zero.

This  approach was initiated by the second author in \cite{veys2}. There he showed that the naive straightforward generalization of the  conjecture is not true in general, and proposed to investigate the following adaptation for normal surface germs. Let $d$ denote the least common denominator of the log discrepancies $\nu_j$, $j \in T_e$ (which is independent of the chosen embedded resolution).
{\em Question: if $s_0$ is a pole of $\Zeta{f}{s}{o}$ or $\Zetam{f}{s}{o}$, is then $e^{2\pi i d s_0}$ a monodromy eigenvalue of $f$ at some point of $f^{-1}\{0\}$ close to $o$?}

Rodrigues proved in \cite{rodrigues} that this question has a positive answer in many cases, for instance in the setting of the theorem below.
But he also gave counterexamples, see {\cite[Examples 3.5 and 3.6]{rodrigues}}
\begin{theorem}[{\cite[Theorem 3.2]{rodrigues}}]\label{mon gen}
	Let $f$ be a non-constant regular function on a normal surface germ $(S,o)$, where $f(o)=0$. Suppose all the log discrepancies $\nu_j$, $j\in T_e,$ are integral. If $s_0 \in \Q_{\leq 0}$ is a pole of $\Zeta{f}{s}{o}$ or $\Zetam{f}{s}{o}$, then $e^{2\pi i s_0}$ is a monodromy eigenvalue of $f$ at some point of $f^{-1}\{0\}$ close to $o$.
\end{theorem}


The second possibility is to consider some \lq natural\rq\ regular differential $2$-form $\omega$ on $S$, generalizing  $dx\wedge dy$ on $\A^2$, and consider
$\Div{h^*(\omega)} = \sum_{j \in T} (\nu_j-1)E_j$.
This was not investigated yet. In Section \ref{singular}, we explore this approach. 

More precisely, a natural generalization of $dx\wedge dy$ is obtained via an embedding of $S$ in some affine space $\A^m$, namely (the restriction to $S$ of)  a form $d\ell_1\wedge d\ell_2$, where $\ell_1$ and $\ell_2$ are generic linear forms on $\A^m$.
In fact, in this setting there is no reason to assume that $S$ is normal (in the first approach, this was motivated by considering the relative canonical divisor). We can as well assume that $S$ is the germ of an isolated surface singularity.

\medskip
Returning to  $\A^2$, zeta functions associated to $f$ and a form $gdx\wedge dy$, with $g$ another polynomial, were already studied in relation with monodromy. For \lq most\rq\ $g$, these zeta functions have poles that do not induce monodromy eigenvalues of $f$. In \cite{nemethiveys}, families of forms $gdx\wedge dy$ were determined for which all poles of the associated zeta functions do induce monodromy eigenvalues of $f$, where the $g$ are given by combinatorial conditions, depending on the minimal embedded resolution graph of $f$.

A not yet investigated question in this setting is whether we can consider a $2$-form on $\A^2$, that is {\em intrinsically associated to the polynomial $f$ itself}, and propose for the zeta functions associated to such a form the analogue of the original monodromy conjecture. Natural intrinsic forms are $g dx\wedge dy$, where $g$ is the hessian of $f$ or a generic polar of $f$. We investigate this in Section \ref{planecurves}.

\medskip
The outcome is that we provide counterexamples to the generalized monodromy conjecture statement in each case, see Sections \ref{singular} and \ref{planecurves}.
In addition, in the case of a generic polar $g$, we discover a remarkable link between poles of the associated zeta function and the intersection behaviour of the strict transform of the polar curve with the exceptional locus of the minimal resolution of $f$. In this context we formulate an open question of independent interest, intrinsically about polar corves.


\section{Preliminaries}

\subsection{Zeta functions}\label{zetafunctions}

Both the topological and motivic zeta function originate from the $p$-adic Igusa zeta function. For more motivation and history, we refer to \cite{veys} and the references therein.
In order to define the motivic zeta function, we first briefly introduce the Grothendieck ring of complex varieties.

\begin{definition}
	The \emph{Grothendieck ring of complex varieties} $\GR$ is the quotient of the free abelian group generated by the symbols $[X]$, where $X$ runs over all complex varieties, by the relations
	$	[X] = [Y] $, if  $X \cong Y$, and
	$	[X] = [Y] + [X \setminus Y]$,  if  $ Y$ is Zariski-closed in  $X$.

	There is a natural commutative ring structure on $\GR$ given by $[X] \cdot [Y] = [X \times Y]$. We set $\LA := [\A^1]$ and $\M:=\GR_\LA$, the ring obtained from $\GR$ by inverting $\LA$.
\end{definition}

	Let $f \in \C[x_1,\hdots,x_n] \setminus \C$, such that   $f(o)=0$.

\begin{definition}  Denote by $C_{m,o}(f)$ the {\em $m$-contact locus of $f$}, being the subvariety of the space of $m$-jets on $\C^n$, attached at the origin, that have contact order precisely $m$ with $\Div{f}$. The (local) motivic zeta function of $f$ is
$$\Zetam{f}{s}{o} := \sum_{m\geq 1} [C_{m,o}(f)] \LA^{-mn} (\LA^{-s})^m
\in \M[[\LA^{-s}]],$$
where $\LA^{-s}$ is a formal variable.
\end{definition}
The notation $\LA^{-s}$ arose from the analogy with the $p$-adic setting; the $p$-adic Igusa zeta function is given by a similar series in $p^{-s}$.
We refer to e.g. \cite{denefloeser2, cns} for details. We will not use the definition in the sequel.

\smallskip
There is a formula for $\Zetam{f}{s}{o}$ in terms of an embedded resolution  $h:Y \rightarrow U$  of the germ at the origin of $\Div{f}\subset \A^n$. Generalizing notation from the introduction, $h$ is a proper birational morphism,  $U$ is a small enough neighbourhood of $o$ in $\A^n$, $Y$ is smooth, the restriction of $h$ to $Y \setminus h^{-1}(f^{-1}\{0\})$ is an isomorphism and $h^{-1}(f^{-1}\{0\})=\cup_{j \in T} E_j$ is a simple normal crossings divisor, i.e., the components $E_j$ are non-singular hypersurfaces that intersect transversally.

 The \emph{numerical data} of each component $E_j$ are $(N_j,\nu_j)$, where $N_j$ and $\nu_j-1$ are the multiplicities of $E_j$ in the divisor defined by $f \circ h$ and $h^*(dx \wedge \hdots \wedge dx_n)$, respectively.
	We also set  $E^\circ_I = \cap_{i \in I} E_i \setminus \cup_{k \notin I} E_k$ for every subset $I$ of $T$.
	Note that $Y$ is the disjoint union of all $E_I^\circ$, $I \subset T$. To simplify notation, we set 
 $E_i^\circ = E_{\{i\}}^\circ$.

\begin{theorem} (\cite{denefloeser2}) Using notation above, we have that
$$	\Zetam{f}{s}{o} := \sum_{I \subset T} [E_I^\circ \cap h^{-1}(0)]\prod_{i \in I} \frac{\LA-1}{\LA^{\nu_i+sN_i}-1}.
$$
\end{theorem}

The topological zeta function of $f$, first introduced in \cite{denefloeser}, turned out to be a specialization of the motivic one, see  \cite[2.3]{denefloeser2}.

\begin{definition}
	The \emph{(local) topological zeta function} of $f$ is
\begin{equation}
		\Zeta{f}{s}{o} := \sum_{I \subset T} \chi(E_I^\circ \cap h^{-1}(0)) \prod_{i \in I}\frac{1}{\nu_i+N_is}  \in \Q(s),
	\end{equation}
where $\chi(.)$ denotes the topological Euler characteristic.
\end{definition}


For $n=2$,  the second author determined in \cite[Theorem 4.3]{veys3} all the poles of these zeta functions.
\begin{theorem}\label{all poles}
	Let $h:Y \rightarrow \A^2$ be the minimal embedded resolution of the germ of $\Div{f}$ at $o$. We have that $s_0$ is a pole of $\Zeta{f}{s}{o}$ or $\Zetam{f}{s}{o}$ if and only if $s_0 = -\frac{\nu_j}{N_j}$ for some component $E_j$ of the strict transform or for some exceptional component $E_j$ intersecting at least three times other components.
\end{theorem}

For some generalizations of these zeta functions for functions on a singular ambient space of arbitrary dimension, see \cite{veys4}.

\subsection{Monodromy conjecture}\label{monconj}
   We first briefly recall the notion of local monodromy \cite{milnor}. 
Let $f$ be a polynomial function $\C^n \rightarrow \C$ and  $p \in f^{-1}\{0\}$. Choose a small enough ball $B \subset \C^n$ with centre $p$ and a small enough punctured disc $D^* \subset \C\setminus \{0\}$ with centre $0$. Then the restriction $f|_{B \cap f^{-1}(D^*)}$ is the \emph{projection} of a \emph{smooth fibration}. `The' \emph{(local) Milnor fibre at} $p$ is the fibre $F_p$ of this fibration. The counterclockwise generator of the fundamental group of $D^*$ induces linear automorphisms $M_i$ on the cohomology vector spaces $H^i(F_p,\C)$. These automorphisms are called the \emph{(algebraic) monodromy actions of} $f$ \emph{at} $p$. A \emph{monodromy eigenvalue} of $f$ at $p$ is an eigenvalue of one of the $M_i$.

It is well known that $H^i(F_p,\C)=0$ for $i \geq n$, and that all monodromy eigenvalues are roots of unity.
There is a classical formula of A'Campo for the alternating product of the characteristic polynomials of the $M_i$ in terms of an embedded resolution of singularities of $f$ \cite{acampo}. We will mention the version that we need in Section \ref{setting}.

We only state the local versions of the conjecture.

\begin{conjecture}[Monodromy conjecture]
	Let $f \in \C[x_1,\hdots,x_n] \setminus \C$ and assume $f(o) = 0$. If $s_0$ is a pole of $\Zeta{f}{s}{o}$ or $\Zetam{f}{s}{o}$, then $e^{2\pi i s_0}$ is a monodromy eigenvalue of $f$ at some point of $f^{-1}\{0\}$ close to $o$.
\end{conjecture}

We note that there is a subtlety concerning the notion of a pole of $\Zetam{f}{s}{o}$, because $\GR$ is not an integral domain \cite{poonen} (even $\LA$ is a zero divisor \cite{borisov}).  We refer to  \cite[Section 4]{bartenveys} for an exact definition.


The motivic zeta function specializes to the topological zeta function. This means that the motivic zeta function is `finer' than the topological zeta function, i.e., every pole of the topological zeta function is also a pole of the motivic zeta function. Hence, the monodromy conjecture in the motivic setting implies it in the topological setting. The other direction is not necessarily true, there exist explicit examples where a pole of the motivic zeta function vanishes for the topological zeta function (such an example is $f=x_1^3+x_2^3+x_3^3+x_4^3+x_5^6$).

\medskip
The conjecture was proven for plane curves by Loeser in 1988 \cite{loeser}. However, he derived this from a stronger statement. Rodrigues gave a more elementary proof in \cite{rodrigues}. In higher dimension, we refer to \cite{veys} for partial results in dimension $3$ and an overview of special families of polynomials, for which the conjecture is proven. In its full generality, the conjecture remains a wide open problem.

\section{Setting}\label{setting}

\subsection{}\label{setting1}
We fix  setting and notation that can be used for both Sections \ref{singular} and \ref{planecurves}.
Let $(S,o)$ be an isolated  (algebraic) surface germ, where $o\in S$ is either smooth or an isolated singularity, and fix a nonconstant regular function germ $f$ on $S$. Let
$\omega$ be a regular $2$-form on $S$.

 Let $h:Y \rightarrow S$ be an embedded resolution of $D:=\Div{f} \cup \Div{\omega}$, where we still assume that
the restriction of $h$ to $Y \setminus  h^{-1}(D) $ is an isomorphism. As before we denote the irreducible components of $h^{-1}(D)$ by $E_j$, $j \in T = T_e \sqcup T_s$, where $T_e$ runs over the exceptional components and $T_s$ over the components of the strict transform of $D$.
We set $E_i^\circ = E_i\setminus \cup_{\ell\in T\setminus\{i\}}E_\ell$ for $i\in T_e$.

The \emph{numerical data} of $E_j$ are $(N_j,\nu_j)$, where $N_j$ and $\nu_j-1$ are the multiplicities of $E_j$ in the divisor defined by $f \circ h$ and $h^*(\omega)$, respectively.
In this more general setting, we still have that  $\nu_j \geq 1$ for all $j\in T$, and $N_j \geq 0$ for all $j\in T$, but now $N_j=0$ for the (strict transform of) components of  $\Div{\omega}$ that are {\em not} a component of $\Div{f}$.

Then the \emph{(local) motivic and topological zeta functions} of $f$ and $\omega$ on $(S,o)$ are
$$\Zetamw{f}{s}{o}{\omega} :=
\sum_{j\in T_e} [E_j^\circ] \frac{\LA-1}{\LA^{\nu_j+N_js}-1} +  \sum_{i \neq j\in T} [E_i \cap E_j]\frac{(\LA-1)^2}{(\LA^{\nu_i +N_is}-1)(\LA^{\nu_j +N_js}-1)} \in \M[[\LA^{-s}]],
$$
and
\begin{equation}\label{ropzetaS}
\Zetaw{f}{s}{o}{\omega}:=
 \sum_{j\in T_e} \frac{\chi(E_j^\circ)}{\nu_j +N_js} + \sum_{i \neq j\in T} \frac{\chi(E_i \cap E_j)}{(\nu_i +N_is)(\nu_j+N_js)} \in \Q(s),
\end{equation}
 respectively.
It is straightforward to verify that these expressions are invariant under blow-ups, see e.g.  \cite{veys2}.


\begin{remark}
	Since the motivic zeta function specializes to the topological zeta function, a counterexample for the topological zeta function is automatically also a counterexample for the motivic zeta function. Hence, we focus on the topological zeta function.
\end{remark}

\subsection{}
Before we explore the adaptations of the conjecture, we study when a candidate pole $\nu_j = -\frac{\nu_j}{N_j}$, $j \in T$, fails to be an actual pole of the topological zeta function.
Note that $s_0$ is a pole of order $2$ of $\Zetaw{f}{s}{o}{\omega}$ if and only if there exist two intersecting components $E_i$ and $E_j$ such that $s_0=-\frac{\nu_i}{N_i}=-\frac{\nu_j}{N_j}$. We know that any candidate pole of order $2$ is an actual pole of order $2$, since $\chi(E_i \cap E_j)>0$ whenever $E_i \cap E_j \neq \emptyset$.

Fix an exceptional component $E_0$, intersecting exactly $r$ times other components $E_1, \hdots, E_r$. Then $\chi(E_0^\circ) = 2-g_0-r$, where $g_0 = g(E_0)$ is the genus of the curve $E_0$. In the formula (\ref{ropzetaS}), 
 the contribution of $E_0$ is
$$\frac{1}{\nu_0+N_0s}\left(2-2g_0-r + \sum_{j=1}^r \frac{1}{\nu_j+N_js}\right).$$\\
Assume now that $\frac{\nu_0}{N_0} \neq \frac{\nu_j}{N_j}$ for $j=1,\hdots,r$; then $s_0 = -\frac{\nu_0}{N_0}$ is a candidate pole of order $1$ of this contribution with residue
\begin{equation} \label{residue}
	R = \frac{1}{N_0}\left(2-2g_0-r + \sum_{j=1}^r \frac{1}{\nu_j-\frac{\nu_0}{N_0}N_j}\right) = \frac{1}{N_0}\left(2-2g_0-r + \sum_{j=1}^r \frac{1}{\alpha_j}\right),
\end{equation}
where $\alpha_j := \nu_j-\frac{\nu_0}{N_0}N_j$ for $j=1,\hdots,r$. We say that $E_0$ {\em does not contribute to the poles of $\Zetaw{f}{s}{o}{\omega}$} if this residue is zero.

\begin{proposition}\label{contribution}
(1)	Using the same notation as above, $E_0$ does not contribute to the poles of $\Zetaw{f}{s}{o}{\omega}$ if $g_0 = 0$ and $r=1$ or $r=2$.  

(2) More generally, $E_0$ does not contribute to the poles of $\Zetaw{f}{s}{o}{\omega}$ if $g_0 = 0$ and $\alpha_j\neq 1$ for at most two $j\in \{1,\hdots,r\}$.
\end{proposition}

\begin{remark}
	The same result holds for $\Zetamw{f}{s}{o}{\omega}$.
\end{remark}

This follows immediately from the lemma below.


\begin{lemma}[{\cite[Lemma 4.1]{veys}}]\label{num lemma}
	Let the exceptional component $E_0$ intersect exactly $r$ times other components $E_1,\hdots,E_r$. Denote $\kappa = -E_0^2$, where $E_0^2=E_0\cdot E_0$ is the self-intersection number of $E_0$ on $Y$, and $g_0 = g(E_0)$ for the genus of $E_0$. Then
	\begin{enumerate}[label=(\arabic*)]
		\item $\kappa N_0 = \sum_{j=1}^r N_j$;
		\item $\kappa \nu_0 = \sum_{j=1}^r (\nu_j-1)+2-2g_0$; and
		\item $\sum_{j=1}^r (\alpha_j-1)=2g_0-2$, where $\alpha_j = \nu_j-\frac{\nu_0}{N_0}N_j$ for $j=1,\hdots,r$.
	\end{enumerate}
\end{lemma}

\subsection{}
For the notion of monodromy in this setting, namely of the function $f$ on the possibly singular surface $S$, 
we embed $(S,p)$ in $(\C^m,p)$ for some sufficiently large $m$. Then $(S,p)$ has defining polynomials $f_1,\hdots,f_{r} \in \C[x_1,\hdots,x_m]$ and we view $f$ in $\C[x_1,\hdots,x_m]/(f_1,\hdots,f_{r})$. So, we can extend $f$ to a map $\C^m \rightarrow \C$.  
 Choose a small enough ball $B \subset \C^m$ with centre $p$ and a small enough punctured disc $D^* \subset \C\setminus \{0\}$ with centre $0$, and then the \emph{(local) Milnor fibre at} $p$ is the fibre $F_p$ of the smooth fibration $f|_{B \cap S \cap f^{-1}(D^*)}$.
 The counterclockwise generator of the fundamental group of $D^*$ again induces linear automorphisms $M_i$ on the cohomology vector spaces $H^i(F_p,\C)$.
Now $H^i(F_p,\C)=0$ for $i \geq \dim S = 2$; a \emph{monodromy eigenvalue} of $f$ at $p$ is an eigenvalue of  $M_0$ or $M_1$.


\begin{theorem}[A'Campo's formula \cite{acampo}.]\label{A campo}
Using notation as above, let $P_{i,p}(t)$ denote the characteristic polynomial of  $M_i$ on $H^i(F_p,\C)$, for $i=0,1$.
Choose an embedded resolution $h:Y \rightarrow S$ of $f^{-1}\{0\}$, for which we use similar notations $E_j$ and $N_j$, $j\in T$, as above. Then
	$$\frac{P_{1,p}(t)}{P_{0,p}(t)} = \prod_{j \in T_e} (t^{N_j}-1)^{-\chi(E_j^\circ \cap h^{-1}(p))}.$$
\end{theorem}

The right hand side is thus an alternating product of cyclotomic polynomials; we will denote in the sequel the $k$th cyclotomic polynomial by $\Phi_k$.

\begin{remark}\label{strict transform poles}
(1) When $S$ is smooth at $p$, it is well known that $P_{0,p}(t)=t^d -1$, where $d$ is the greatest common divisor of the $N_j, j\in T_s$.

(2) In particular, in our setting, as fixed in Subsection \ref{setting1}, we can take $j \in T_s$ and a point $p \in S$, {\em different from $o$}, on the component of $f^{-1}\{0\}$ whose strict transform by $h$ is $E_j$. By A'Campo's formula
and (1),  every $N_j$th root of unity is a monodromy eigenvalue of $f$ at $p$. In particular, if $s_0 = -\frac{\nu_j}{N_j}$ is a pole of $\Zetaw{f}{s}{o}{\omega}$, then $e^{2\pi i s_0}$ is always a monodromy eigenvalue of $f$. Hence, we are not really interested in the poles arising from components of the strict transform and we will not discuss them in the examples.
We call those the {\em trivial candidate poles}.

(3) When $S$ is singular at $o$, we still have that $P_{0,o}(t)$ has the form $t^d -1$, but then in general we only know  that $d$ {\em divides} the greatest common divisor of the $N_j, j\in T_s$, see \cite[3.4.8]{nemethi}.

(4) In our all examples below, we have however that $\gcd_{j\in T_s} N_j=1$, and then $P_{0,o}(t)=t -1$ and $P_{1,o}(t)/P_{0,o}(t)$ really detects {\em all} eigenvalues different from $1$.
\end{remark}


\section{Curves on singular surfaces}\label{singular}

Since every variety is locally affine, let $(S,o)$ be embedded in $(\A^m,o)$, where $m$ is sufficiently large. Consider on $\A^m$ the $2$-form  $dl_1 \wedge dl_2$, where $l_1,l_2$ are generic linear functions on $\A^m$.  We take as $2$-form $\omega$ on $S$ the restriction of $dl_1 \wedge dl_2$ to $S$.

\begin{remark}
This construction indeed generalizes the form $dx\wedge dy$ on $S=\A^2$: any generic such form $dl_1 \wedge dl_2$ on $\A^2$ itself is of the form $\lambda dx \wedge dy$, for some $\lambda \in \C^*$.
\end{remark}

\subsection{Curves on normal surfaces}
In this subsection, we consider a {\em normal} surface germ $(S,o)$.
 Simple examples of normal surface singularities are the ADE singularities, see e.g. \cite{wall}. These are surfaces in $\A^3$, so we use this embedding to obtain an embedded resolution and use the coordinates $x,y$ and $z$ on $\A^3$. For all the examples in this subsection, we can write the $2$-form $\omega$ as
$$\omega = \alpha dx \wedge dy + \beta dx \wedge dz + \gamma dy \wedge dz,$$
where $\alpha$, $\beta$ and $\gamma$ are generic complex coefficients.


\begin{example}
Take as $(S,o)$ the $A_1$ singularity, defined by the polynomial $xy-z^2$, and as $f$ the function $x^4-y$ on $S$. We start with a blow-up of $\A^3$ in the origin to compute the minimal resolution of the $A_1$ singularity.
\medskip
	
	\underline{Blow-up 1.}
	All the relevant information is in the chart given by $(x,y,z) \mapsto (x,xy,xz)$. The total transform of the surface is given by $x^2(y-z^2)$. Hence, its strict transform is $V(y-z^2)$ and the exceptional component is $V(x)$. Here and in the sequel, $V(I)$ is the zero-locus of the polynomial ideal $I$. The surface $V(y-z^2)$ is non-singular and is isomorphic to $\A_{x,z}^2$ by substituting $y$ by $z^2$. Hence the pullback of $f$ is $x(x^3-y) = x(x^3-z^2)$. We compute the pullback of $\omega$ by the following steps:
	\begin{align*}
			dx \wedge dy & \leadsto x dx \wedge dy\\
			& \leadsto x dx \wedge d(z^2) = 2x zdx \wedge dz,\\
			dx \wedge dz & \leadsto x dx \wedge dz,\\
			dy \wedge dz &  \leadsto (ydx+xdy) \wedge (zdx+xdz) = xy dx \wedge dz -xz dx \wedge dy +x^2 dy \wedge dz\\
			& \leadsto -xz^2 dx \wedge dz.
	\end{align*}
	We conclude that the pullback of $\omega$ is
	$$x(2\alpha z +\beta -\gamma z^2)dx \wedge dz.$$
	Note that the exceptional curve $E_1$ and $V(2\alpha z +\beta -\gamma z^2)$ intersect twice, see Figure \ref{figure22}. In this chart, $V(x^3-z^2)$ still has a singular point, so further blow-ups are needed. This computation is classical and well known. The numerical data can be calculated by the well-known algorithm in Proposition \ref{algoritme} below, see e.g. {\cite[Section 8.3]{wall2}}.
	
	\begin{proposition}\label{algoritme}
		Let $f \in \C[x,y]\setminus \C$, let $\omega =gdx \wedge dy$ with $g \in \C[x,y] \setminus \{0\}$, and let an embedded resolution $h$ of $V(fg)$ be given by the composition of some blow-ups $\pi_1,\hdots,\pi_m$. Assume we already performed the first $k < m$ blow-ups, creating the exceptional components $E_1,\hdots, E_k$ with their numerical data $(N_1,\nu_1),\hdots, (N_k,\nu_k)$. Denote the multiplicity of the strict transform of $V(f)$ and $V(g)$ in the centre of the next blow-up $\pi_{k+1}$ by $\mu_0$ and $\lambda_0$, respectively. There are three distinct and exhaustive cases for the next blow-up $\pi_{k+1}$.
		
		\begin{enumerate}[label=(\arabic*)]
			\item Blow up in a point that does not lie on any exceptional component. (This means that k=0.) Then $N_{k+1} = \mu_0$ and $\nu_{k+1} = \lambda_0+2$.
			\item Blow up in a point that lies on exactly one exceptional component $E_i$. Then $N_{k+1} = N_k+\mu_0$ and $\nu_{k+1} = \nu_i+\lambda_0+1$.
			\item Blow up in a point that lies on exactly two exceptional components $E_i$ and $E_j$. Then $N_{k+1} = N_i+N_j+\mu_0$ and $\nu_{k+1} = \nu_i+\nu_j+\lambda_0$.
		\end{enumerate}
	\end{proposition}
	
 In the present example, we have thus that $g = 2\alpha z +\beta -\gamma z^2$.
	Let $E_1,E_2,E_3$ and $E_4$ be the consecutive exceptional components and $E_0$ the strict transform of $V(f)$. In Figure \ref{figure22}, the process of the different blow-ups is depicted, showing in particular how the different components intersect. The components of the strict transform of $\Div{\omega}$ are indicated by a dashed blue line.  Also, the numerical data $(N_j,\nu_j)$ for every component $E_j$ are indicated on the final graph.
	\medskip
	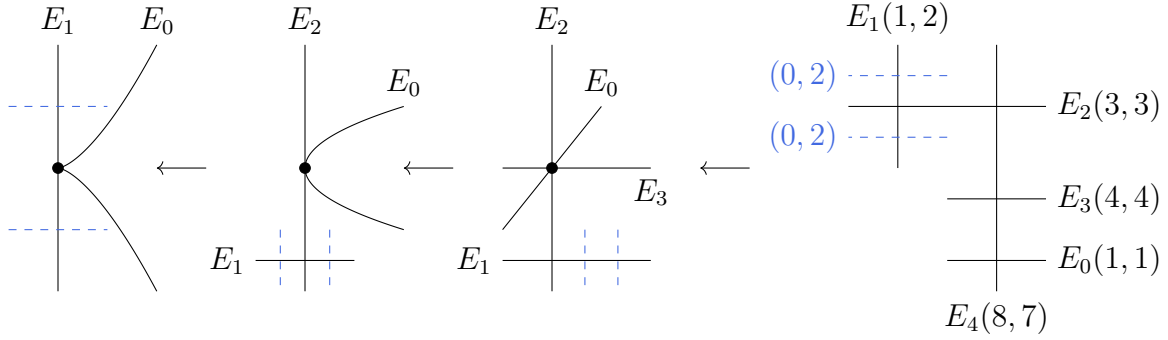
\begin{figure}[h]
		\centering
		\begin{tikzpicture}
			\begin{axis}[
				width=18cm,
				height=6.5cm,
				xmin=-1, xmax=11.5,
				ymin=-1.5, ymax=1.5,
				samples=100,
				hide axis,
				]
				\addplot [domain=-1:1,black] (x^2,x^3) node[above] {$E_0$};
				\addplot [only marks] table {
					0 0
				};
				\addplot [black] table {
					0 -1
					0 1
				} node[above] {$E_1$};
				\addplot [dashed,hot] table {
					-0.5 -0.5
					0.5 -0.5
				};
				\addplot [dashed,hot] table {
					-0.5 0.5
					0.5 0.5
				};
				\addplot [<-] table {
					1 0
					1.5 0
				};

				\addplot [only marks] table {
					2.5 0
				};
				\addplot [domain=-0.5:0.5,black] (4*x^2+2.5,x) node[pos=1,above] {$E_0$};
				\addplot [black] table {
					2.5 -1
					2.5 1
				} node[above] {$E_2$};
				\addplot [black] table {
					3 -0.75
					2 -0.75
				} node[left] {$E_1$};
				\addplot [dashed,hot] table {
					2.75 -0.5
					2.75 -1
				};
				\addplot [dashed,hot] table {
					2.25 -0.5
					2.25 -1
				};
				\addplot [<-] table {
					3.5 0
					4 0
				};
				
				\addplot [only marks] table {
					5 0
				};
				\addplot [domain=-0.5:0.5,black] (x+5,x) node[above] {$E_0$};
				\addplot [black] table {
					5 -1
					5 1
				} node[above] {$E_2$};
				\addplot [black] table {
					4.5 0
					6 0
				} node[below] {$E_3$};
				\addplot [black] table {
					6 -0.75
					4.5 -0.75
				} node[left] {$E_1$};
				\addplot [dashed,hot] table {
					5.333 -0.5
					5.333 -1
				};
				\addplot [dashed,hot] table {
					5.667 -0.5
					5.667 -1
				};
				
				\addplot [<-] table {
					6.5 0
					7 0
				};
				\addplot [black] table {
					9 -0.25
					10 -0.25
				} node[right] {$E_3 (4,4)$};
				\addplot [black] table {
					8 0.5
					10 0.5
				} node[right] {$E_2 (3,3)$};
				\addplot [black] table {
					8.5 0
					8.5 1
				} node[above] {$E_1 (1,2)$};
				\addplot [dashed,hot] table {
					9 0.75
					8 0.75
				} node[left] {$(0,2)$};
				\addplot [dashed,hot] table {
					9 0.25
					8 0.25
				} node[left] {$(0,2)$};
				\addplot [black] table {
					9 -0.75
					10 -0.75
				} node[right] {$E_0 (1,1)$};
				\addplot [black] table {
					9.5 1
					9.5 -1
				} node[below] {$E_4 (8,7)$};
			\end{axis}
		\end{tikzpicture}
		\caption{Process of an embedded resolution of $x^4-y$}
		\label{figure22}
	\end{figure}
	
	There are no poles of order $2$ and, according to Proposition \ref{contribution}, two non-trivial candidate poles, namely $s_0 = -\frac{\nu_1}{N_1} = -2$ and $s_0 = \frac{\nu_4}{N_4} = -\frac{7}{8}$. It is clear that $e^{2\pi i s_0}=1$ is the trivial monodromy eigenvalue for $s_0 = -2$; so it is not important whether $s_0 = -2$ is an actual pole. We have that $s_0=-\frac{7}{8}$ is a pole, since its residue equals
	$$\frac{1}{N_4}\left(\frac{1}{\alpha_0}+\frac{1}{\alpha_2}+\frac{1}{\alpha_3}-1\right) = \frac{35}{24},$$
	where
	$$\alpha_0 = \nu_0-\frac{7}{8} \cdot N_0= \frac{1}{8}, \quad \alpha_2 =\nu_2-\frac{7}{8}\cdot N_2=\frac{3}{8}, \quad \alpha_3=\nu_3-\frac{7}{8}\cdot N_3=\frac{1}{2}.$$
	We compute the monodromy eigenvalues of $f$ using A'Campo's formula (Theorem \ref{A campo}):
	\begin{align*}
		\prod_{j =1}^4 (t^{N_j}-1)^{-\chi(E_j^\circ)} & = \frac{t^{8}-1}{(t-1)(t^4-1)} = \frac{\Phi_1\Phi_2\Phi_4\Phi_8}{(\Phi_1) \cdot (\Phi_1\Phi_2\Phi_4)} = \frac{\Phi_8}{\Phi_1}.
	\end{align*}
	From this, we conclude that $e^{2\pi i s_0}$ is a monodromy eigenvalue of $f$ for every pole $s_0$. Now, this could be a coincidence. The candidate pole arising from $E_1$ induces the trivial mondromy eigenvalue. Let us consider an example, where this is not the case.
\end{example}

\begin{example}
Take again as $(S,o)$ the $A_1$ singularity, defined by the polynomial $xy-z^2$, and now as $f$ the function $y^{10}-z^{13}$. We repeat the process from the previous example to obtain an embedded resolution.
\medskip

\underline{Blow-up 1.}
All the relevant information is in the chart given by $(x,y,z) \mapsto (x,xy,xz)$. Substituting in $f$ gives $x^{10}(y^{10}-x^3z^{13}) = x^{10}(z^{20}-x^3z^{13}) = x^{10}z^{13}(z^{7}-x^3)$. There are two components of the strict transform: $E_0=V(z^{7}-x^3)$ and $E_0' = V(z)$. The pullback of $\omega$ is
$$x(2\alpha z +\beta -\gamma z^2)dx \wedge dz.$$
Then the numerical data of $E_0$, $E_0'$ and $E_1$ are $(1,1)$, $(13,1)$ and $(10,2)$, respectively. Note that $E_1$ and $V(2\alpha z +\beta -\gamma z^2)$ intersect twice, see Figure \ref{figure24}.  In this chart, $V(z^7-x^3)$ still has a singular point, so further blow-ups are needed. The intersection diagram for an embedded resolution is shown in Figure \ref{figure24}.
\medskip
	
	\begin{figure}[h!]
		\centering
		\begin{tikzpicture}
			\begin{axis}[domain=-1:1,
				width=18cm,
				height=6.5cm,
				xmin=-1, xmax=11.5,
				ymin=-1.5, ymax=1.5,
				samples=100,
				hide axis,
				]
				\addplot [only marks] table {
					0 0
				};
				\addplot [black] table {
					0 -1
					0 1
				} node[above] {$E_1$};
				\addplot [domain=-1:1,black] (x^2,x^3) node[above] {$E_0$};
				\addplot [black] table {
					0.5 0
					-0.5 0
				} node[left] {$E_0'$};
				\addplot [dashed,hot] table {
					-0.5 -0.5
					0.5 -0.5
				};
				\addplot [dashed,hot] table {
					-0.5 0.5
					0.5 0.5
				};
				\addplot [<-] table {
					1.5 0
					2 0
				};
				
				\addplot [black] table {
					4.5 -0.75
					7.5 -0.75
				} node[right] {$E_6 (130,17)$};
				\addplot [black] table {
					4.75 1
					4.75 -1
				} node[below] {$E_5 (90,12)$};
				\addplot [black] table {
					7.25 1
					7.25 -1
				} node[below] {$E_3 (39,5)$};
				\addplot [black] table {
					6 -1
					6 1
				} node[above] {$E_0 (1,1)$};
				\addplot [black] table {
					5 0.75
					3.75 0.75
				} node[left] {$E_4 (50,7)$};
				\addplot [black] table {
					7 0.75
					8.25 0.75
				} node[right] {$E_2 (26,3)$};
				\addplot [black] table {
					8 0.5
					8 1
				} node[above] {$E_0' (13,1)$};
				\addplot [black] table {
					4 -0.25
					4 1
				} node[above] {$E_1 (10,2)$};
				\addplot [dashed,hot] table {
					4.25 0
					3.75 0
				} node[left,hot] {$(0,2)$};
				\addplot [dashed,hot] table {
					4.25 0.375
					3.75 0.375
				} node[left,hot] {$(0,2)$};
			\end{axis}
		\end{tikzpicture}
		\caption{Intersection diagram for $y^{10}-z^{13}$}
		\label{figure24}
	\end{figure}
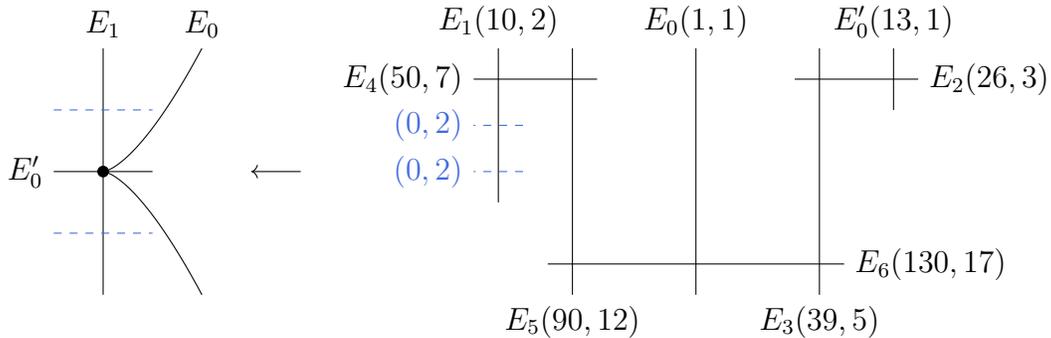
	
The residue of the candidate pole $s_0 = -\frac{\nu_1}{N_1} = -\frac{1}{5}$ induced by $E_1$ is easily calculated as $ -\frac{1}{30}$, hence it is a pole. We compute the monodromy eigenvalues of $f$ using A'Campo's formula (Theorem \ref{A campo}):
	\begin{align*}
		\prod_{j =1}^6 (t^{N_j}-1)^{-\chi(E_j^\circ)} & = \frac{t^{130}-1}{t^{10}-1} = \Phi_{13}\Phi_{26}\Phi_{65}\Phi_{130}.
	\end{align*}
We also have the trivial monodromy eigenvalue $1$. From this, we conclude that $e^{2\pi i s_0}$ is not a monodromy eigenvalue of $f$ for the pole $s_0=-\frac{1}{5}$. This example provides a counterexample for a possible generalization of the monodromy conjecture in this setting.
\end{example}

\begin{remark}
	Let us mention a few other examples of normal surfaces. On the $A_2$ singularity in $\A^3$, given by the polynomial $xy-z^3$, we consider $f_1 = y^3-x^4$ and $f_2 = z^7-x^5$ as  regular functions. For $f_1$ the generalization of the monodromy conjecture is true, whereas for $f_2$ this is not the case.

As a last example, we consider a cyclic quotient singularity, that is not a hypersurface or a complete intersection (as in the previous examples). The surface in $\A^4$ (with coordinates $x_1,x_2,x_3,x_4$) is defined by the polynomials $x_2x_3-x_1x_4$, $x_2^2-x_1x_3$ and $x_3^2-x_2x_4$. Consider $$\omega = \alpha_1 dx_1\wedge dx_2+\alpha_2 dx_1 \wedge dx_3+\alpha_3 dx_1 \wedge dx_4 + \alpha_4 dx_2 \wedge dx_3+\alpha_5dx_2 \wedge dx_4+\alpha_6dx_3\wedge dx_4,$$
	where the $\alpha_i$ are generic complex coefficients. For $f_1=x_1^5-x_2^3$ the generalization of the monodromy conjecture is true, but for $f_2=x_3^{11}-x_4^8$ this is not the case. We leave the computations to the reader.
\end{remark}

These examples confirm that there is no possible generalization in this setting of the monodromy conjecture for curves on normal surfaces.

\subsection{Curves on a non-normal surface}
It is also interesting to consider a non-normal surface in our generalized setting. (In the context of the other generalization, mentioned in the introduction, it was necessary that the surface was normal to define the log discrepancies.)

We consider the example \cite[Example 5.8]{cherik} of a non-normal surface $(S,o)$ embedded in $\A^5$ with an isolated singularity at the origin. We use coordinates $z_1,\hdots,z_5$ on $\A^5$. The surface $(S,o)$ is defined by the following three equations:
$$z_1^2 =z_2^3, \quad z_1z_5=z_2z_3z_4 \quad \text{ and }\quad  z_4^3=z_5^2.$$
A resolution of $(S,o)$ is given by $F:\A^2 \rightarrow S:(x,y) \mapsto (y^3,y^2,xy,x^2,x^3)$.
On this surface $(S,o)$ we consider the linear $2$-form
\begin{align*}
	\omega & = \alpha_1 dz_1 \wedge dz_2 + \alpha_2 dz_1 \wedge dz_3 + \alpha_3 dz_1 \wedge dz_4 + \alpha_4 dz_1 \wedge dz_5 + \alpha_5 dz_2 \wedge dz_3\\
	& + \alpha_6 dz_2 \wedge dz_4 + \alpha_7 dz_2 \wedge dz_5 + \alpha_8 dz_3 \wedge dz_4 + \alpha_9 dz_3 \wedge dz_5 + \alpha_{10} dz_4 \wedge dz_5,
\end{align*}
where the $\alpha_i$ are generic complex coefficients.

\begin{example}
	Take $f=z_5-z_2$. The inverse image of $V(f)$ on $(S,o)$ by $F$ is defined by $x^3-y^2$. We compute the pullback $F^*\omega$ via
	\begin{align*}
		dz_1 \wedge dz_2 & \leadsto d(y^3) \wedge d(y^2) = (3y^2dy) \wedge (2ydy) = 0,\\
		dz_1 \wedge dz_3 & \leadsto (3y^2dy) \wedge (ydx+xdy) = -3y^3dx \wedge dy,\\
		dz_1 \wedge dz_4 & \leadsto (3y^2dy) \wedge (2xdx) = -6xy^2dx \wedge dy,\\
		dz_1 \wedge dz_5 & \leadsto (3y^2dy) \wedge (3x^2dx) = -9x^2y^2dx \wedge dy,\\
		dz_2 \wedge dz_3 & \leadsto (2ydy) \wedge (ydx+xdy) = -2y^2dx \wedge dy,\\
		dz_2 \wedge dz_4 & \leadsto (2ydy) \wedge (2xdx) = -4xydx \wedge dy,\\
		dz_2 \wedge dz_5 & \leadsto (2ydy) \wedge (3x^2dx) = -6x^2ydx \wedge dy,\\
		dz_3 \wedge dz_4 & \leadsto (ydx+xdy) \wedge (2xdx) = -2x^2dx \wedge dy,\\
		dz_3 \wedge dz_5 & \leadsto (ydx+xdy) \wedge (3x^2dx) = -3x^3dx \wedge dy,\\
		dz_4 \wedge dz_5 & \leadsto (2xdx) \wedge (3x^2dx) =0,
	\end{align*}
	such that this pullback equals
\begin{equation}\label{pullbackform}
	-(3\alpha_2y^3+6\alpha_3xy^2+9\alpha_4x^2y^2+2\alpha_5y^2+4\alpha_6xy+6\alpha_7x^2y+2\alpha_8x^2+3\alpha_9x^3)dx \wedge dy.
\end{equation}
	We have not yet obtained an embedded resolution, since $V(x^3-y^2)$ is still singular in the origin of $\A^2$. After further blow-ups, we obtain an embedded resolution as shown in Figure \ref{figure16}.
	There are no poles of order $2$ and two non-trivial candidate poles, namely $s_0=-\frac{\nu_1}{N_1} = -2$ and $s_0 = -\frac{\nu_3}{N_3} = -\frac{3}{2}$. One can calculate that the residue of $s_0 = -\frac{3}{2}$ is zero and that the one of $s_0 = -2$ is non-zero. The pole $s_0 = -2$ induces the trivial monodromy eigenvalue $1$.
	
\end{example}
\begin{figure}[h]
	\centering
	\begin{subfigure}[t]{0.48\textwidth}
		\centering
		\begin{tikzpicture}
			\begin{axis}[
				width=2.5\textwidth,
				height=7cm,
				xmin=-1, xmax=10,
				ymin=-1.5, ymax=1.5,
				samples=100,
				hide axis,
				]
				\addplot [black] table {
					0 0
					2 0
				} node[right] {$E_1 (2,4)$};
				\addplot [black] table {
					0 0.7
					2 0.7
				} node[right] {$E_0 (1,1)$};
				\addplot [black] table {
					0 -0.7
					2 -0.7
				} node[right] {$E_2 (3,5)$};
				\addplot [black] table {
					1 1
					1 -1
				} node[below] {$E_3 (6,9)$};
				\addplot [dashed,hot] table {
					0.3 0.3
					0.3 -0.3
				} node[below,hot] {$(0,2)$};
				\addplot [dashed,hot] table {
					0.6 -0.3
					0.6 0.3
				} node[above,hot] {$(0,2)$};
			\end{axis}
		\end{tikzpicture}
		\caption{For $z_5 - z_2$}
		\label{figure16}
	\end{subfigure}
	\hfill
	\begin{subfigure}[t]{0.48\textwidth}
		\centering
		\begin{tikzpicture}
			\begin{axis}[domain=-1:1,
				width=2.5\textwidth,
				height=7cm,
				xmin=-1, xmax=10,
				ymin=-1.5, ymax=1.5,
				samples=100,
				hide axis,
				]
				\addplot [black] table {
					0 -0.75
					2 -0.75
				} node[right] {$E_0 (1,1)$};
				\addplot [black] table {
					0 0
					2 0
				} node[right] {$E_3 (8,9)$};
				\addplot [black] table {
					1.5 0.25
					1.5 -0.25
				} node[below] {$E_2 (4,5)$};
				\addplot [black] table {
					1 1
					1 -1
				} node[below] {$E_4 (12,13)$};
				\addplot [black] table {
					0 0.75
					2 0.75
				} node[right] {$E_1 (3,4)$};
				\addplot [dashed,hot] table {
					0.6 0.45
					0.6 1.05
				} node[above,hot] {$(0,2)$};
				\addplot [dashed,hot] table {
					0.3 1.05
					0.3 0.45
				} node[below,hot] {$(0,2)$};
			\end{axis}
		\end{tikzpicture}
		\caption{For $z_1 - z_4^2$}
		\label{figure17}
	\end{subfigure}
	\caption{Intersection diagrams}
	\label{combinedfig}
\end{figure}
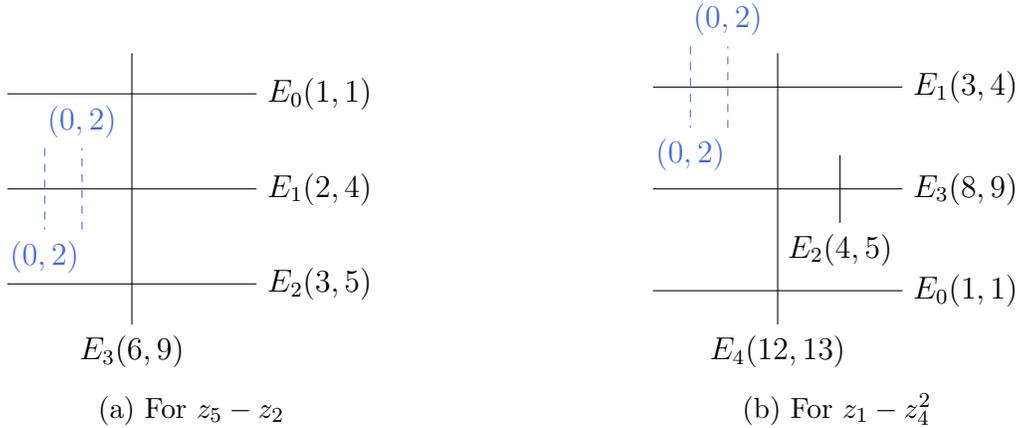

\begin{example}
	Let us consider a more interesting example, namely $f = z_1-z_4^2$. The strict transform is $V(y^3-x^4)$ on $\A^2$. The calculations for the pullback $F^* \omega$ are exactly the same as in the previous example, we obtain
the form (\ref{pullbackform}).
	After further blow-ups, we obtain the resolution graph as pictured in Figure \ref{figure17}. We are interested in the candidate pole $s_0 = -\frac{\nu_1}{N_1} = -\frac{4}{3}$. Its residue is easily computed to be non-zero, hence it is a pole. We compute the monodromy eigenvalues of $f$ using A'Campo's formula (Theorem \ref{A campo}):
	\begin{align*}
		\prod_{j =1}^4 (t^{N_j}-1)^{-\chi(E_j^\circ)} & = \frac{t^{12}-1}{(t^3-1)(t^4-1)} = \frac{\Phi_6 \Phi_{12}}{\Phi_1}.
	\end{align*}
	This means that the pole $s_0 = -\frac{4}{3}$ does not induce a monodromy eigenvalue. We conclude that there is again no possible generalization of the monodromy conjecture in this case.
	
\end{example}

\section{Plane curves}\label{planecurves}
We now take the affine plane as surface $S$ and functions $f \in \C[x,y] \setminus \C$. We will consider on $\A^2$ natural forms $g dx\wedge dy$, where $g$ is intrinsically associated to $f$. First we explore the case when $g$ is the hessian of $f$, and then when $g$ is a  generic polar of $f$.

\subsection{Hessian}
We take $\omega= d\left(\diffp{f}{x}\right) \wedge d\left(\diffp{f}{y}\right) = \text{Hess}(f) dx \wedge dy$, where

$$ \text{Hess}(f) =
 \begin{vmatrix}
		\diffp[2]{f}{x} & \diffp{f}{{x}{y}}\\
		\diffp{f}{{y}{x}} & \diffp[2]{f}{y}
	\end{vmatrix}.
$$

\begin{example}
Take  $f = x^3-y^2$; then
	$g=\text{Hess}(f)= -12x$.
	Hence, we work with the $2$-form $\omega = -12x dx \wedge dy$. The minimal embedded resolution of the cusp is already an embedded resolution of $V(fg)$. We have the intersection diagram as in Figure \ref{figure18}. Observe that there are no poles of order $2$ and one non-trivial candidate pole, namely $s_0 = -\frac{\nu_3}{N_3} = -\frac{7}{6}$. Compute
	$$\alpha_0 = \nu_0-\frac{7}{6} \cdot N_0 = -\frac{1}{6},\quad \alpha_1 = \nu_1-\frac{7}{6} \cdot N_1 = \frac{2}{3}, \quad \alpha_2 = \nu_2-\frac{7}{6} \cdot N_2 = \frac{1}{2},$$
	such that the residue of $s_0 = -\frac{7}{6}$ equals
	$\frac{1}{N_3} \left(\frac{1}{\alpha_0} + \frac{1}{\alpha_1}+ \frac{1}{\alpha_2}-1\right) = -\frac{7}{12}$.
	Thus, for each pole $s_0 $ we have that  $e^{2\pi i s_0}$ is a monodromy eigenvalue of $f$, as $1$ and the two primitive sixth roots of unity are monodromy eigenvalues.
	
	\begin{figure}[h]
		\centering
		\begin{tikzpicture}
			\begin{axis}[
				width=20cm,
				height=6.37cm,
				xmin=-1, xmax=11.5,
				ymin=-1.5, ymax=1.5,
				samples=100,
				hide axis,
				]
				\addplot [black] table {
					2.5 0
					4.5 0
				} node[right] {$E_1 (2,3)$};
				\addplot [black] table {
					2.5 0.7
					4.5 0.7
				} node[right] {$E_0 (1,1)$};
				\addplot [black] table {
					2.5 -0.7
					4.5 -0.7
				} node[right] {$E_2 (3,4)$};
				\addplot [black] table {
					3.5 1
					3.5 -1
				} node[below] {$E_3 (6,7)$};
				\addplot [dashed,hot] table {
					3 0.3
					3 -0.3
				} node[below,hot] {$(0,2)$};
			\end{axis}
		\end{tikzpicture}
		\caption{Intersection diagram for $x^3-y^2$}
		\label{figure18}
	\end{figure}
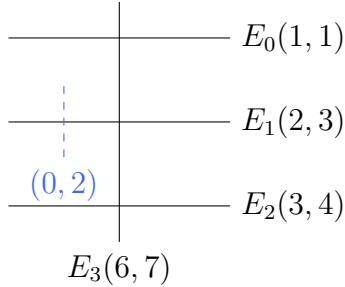
\end{example}

\begin{example}
	Take  $f = y^4-2x^3y^3+x^7-x^6y$;  then we can compute that
	\begin{align*}
		g= \text{Hess}(f) = -36x(x^9+14x^7y-4x^6y^2-14x^4y^2+5x^3y^4+10x^3y^3+4y^5)
	\end{align*}
	such that $\omega = -36x(x^9+14x^7y-4x^6y^2-14x^4y^2+5x^3y^4+10x^3y^3+4y^5) dx \wedge dy$. For this example, the calculations are more difficult. The minimal embedded resolution of $V(f)$, shown in Figure \ref{figure27}, is not an embedded resolution of $V(fg)$ yet. We need to perform further blow-ups, creating $E_6,E_7$ and $E_8$,  obtaining the intersection diagram as in Figure \ref{figure28}. We are interested in the candidate pole arising from $E_8$, namely  $s_0 = -\frac{\nu_8}{N_8}=-\frac{15}{8}$. Its residue equals $\frac{5}{96}$, hence it is a pole. Using A'Campo's formula (Theorem \ref{A campo}), we calculate the monodromy eigenvalues:
	\begin{align*}
		\prod_{j =1}^5 (t^{N_j}-1)^{-\chi(E_j^\circ)} & = \frac{t^{28}-1}{(t^4-1)(t^7-1)} = \frac{\Phi_{14}\Phi_{28}}{\Phi_1}.
	\end{align*}
	Note that we could use both the intersection diagrams (from Figure \ref{figure27} and Figure \ref{figure28}) to calculate the monodromy eigenvalues. Hence, $e^{2\pi i s_0}$ is not a monodromy eigenvalue for the pole $s_0 = -\frac{15}{8}$. This example provides a counterexample to a possible adaptation of the monodromy conjecture for the Hessian $2$-form.
	
	\begin{figure}[h]
		\centering
		\begin{subfigure}[b]{\textwidth}
			\centering
			\begin{tikzpicture}
				\begin{axis}[domain=-1:1, restrict y to domain=-1:1,
					width=20cm,
					height=6.37cm,
					xmin=-1, xmax=10,
					ymin=-1.5, ymax=1.5,
					samples=100,
					hide axis,
					]
					\addplot [black] table {
						3.5 -0.5
						8.5 -0.5
					} node[right] {$E_5$};
					\addplot [black] table {
						4 1
						4 -1
					} node[below] {$E_4$};
					\addplot [black] table {
						6 1
						6 -1
					} node[below] {$E_0$};
					\addplot [black] table {
						8 1
						8 -1
					} node[below] {$E_2$};
					\addplot [black] table {
						3.3 0
						4.3 1
					} node[above] {$E_3$};
					\addplot [black] table {
						3.7 0
						2.7 1
					} node[above] {$E_1$};
				\end{axis}
			\end{tikzpicture}
			\caption{Intersection diagram for $y^4-2x^3y^3+x^7-x^6y$}
			\label{figure27}
		\end{subfigure}
		\medskip
		
		\begin{subfigure}[b]{\textwidth}
			\centering
			\begin{tikzpicture}
				\begin{axis}[domain=-1:1, restrict y to domain=-1:1,
					width=20cm,
					height=6.37cm,
					xmin=-1, xmax=10,
					ymin=-1.5, ymax=1.5,
					samples=100,
					hide axis,
					]
					\addplot [black] table {
						3.5 -0.5
						8.5 -0.5
					} node[right] {$E_5 (28,45)$};
					\addplot [black] table {
						4 1
						4 -1
					} node[below] {$E_4 (20,33)$};
					\addplot [black] table {
						6 1
						6 -1
					} node[below] {$E_0 (1,1)$};
					\addplot [black] table {
						8 1
						8 -1
					} node[below] {$E_2 (7,12)$};
					\addplot [black] table {
						8.5 0
						7.5 1
					} node[above] {$E_7 (14,27)$};
					\addplot [black] table {
						8.2 0
						9.2 1
					} node[above] {$E_6 (7,14)$};
					\addplot [black] table {
						3.3 0
						4.3 1
					} node[above] {$E_3 (12,21)$};
					\addplot [black] table {
						3.7 0
						2.7 1
					} node[above] {$E_8 (16,30)$};
					\addplot [black] table {
						3.2 1
						2.2 0
					} node[below] {$E_1 (4,8)$};
					\addplot [domain=0.25:-0.25,dashed,hot] (x+2.5,-x+0.25) node[above,hot] {$(0,2)$};
					\addplot [domain=-0.25:0.25,dashed,hot] (x+3.3,x+0.5) node[above,hot] {$(0,2)$};
					\addplot [domain=0.25:-0.25,dashed,hot] (x+7.6,x+0.75) node[below,hot] {$(0,2)$};
				\end{axis}
			\end{tikzpicture}
			\caption{Intersection diagram for $y^4-2x^3y^3+x^7-x^6y$ and its Hessian}
			\label{figure28}
		\end{subfigure}
		\caption{Comparison of intersection diagrams}
		\label{fig:intersection-comparison}
	\end{figure}
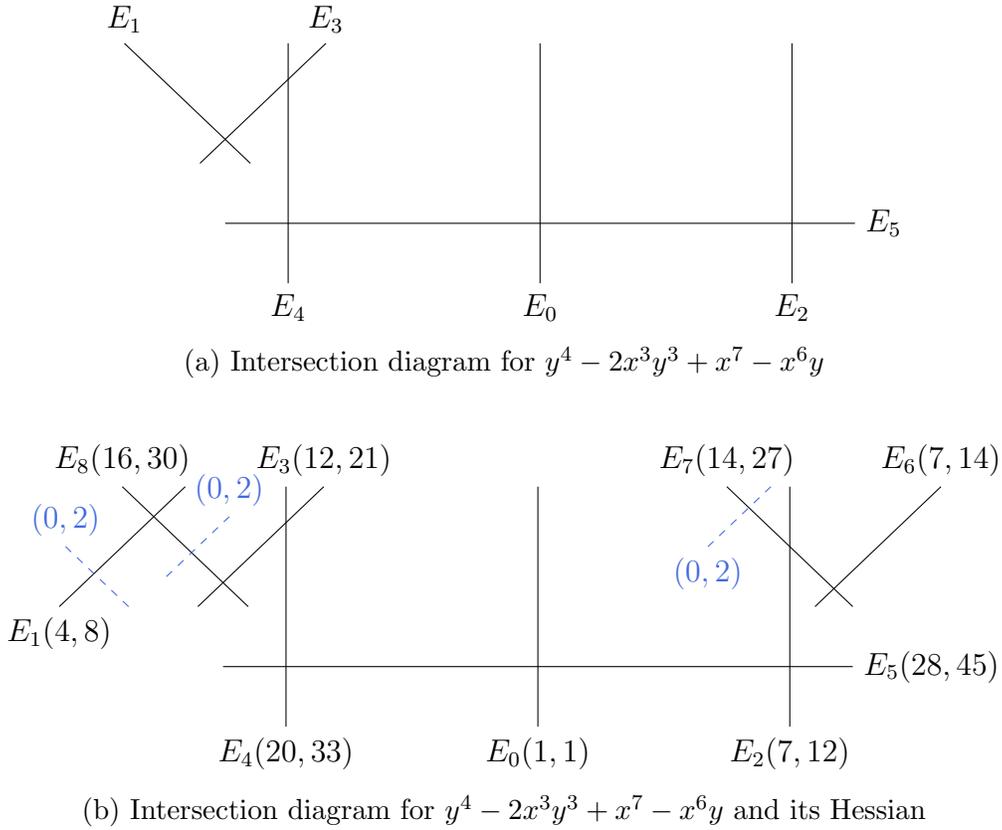
\end{example}

\subsection{Polar curves} \label{polar curves}

A very natural possible adaptation of the monodromy conjecture is using the $2$-form $\omega = df \wedge dl$, where  $l$ is a  generic linear function.
Writing $l=\alpha x + \beta y$,  where $\alpha$ and $\beta$ are generic complex coefficients, we have that
$$df \wedge dl = \left(\diffp{f}{x}dx +\diffp{f}{y}dy\right) \wedge (  \alpha dx + \beta dy)
 =\left(\beta \diffp{f}{x} - \alpha \diffp{f}{y}\right) dx \wedge dy.$$
The curve, appearing as the zero set of $df \wedge dl$, is the well-known polar curve of $f$ with respect to $l$. In the sequel, we call this curve \lq the polar curve\rq, without mentioning the generic linear function $l$, to simplify terminology.

\begin{example}\label{ordinary cusp}
	Take $f = x^3-y^2$;  then $\omega = df \wedge dl = (3\beta x^2+2\alpha y)dx \wedge dy$. Let $g= 3\beta x^2+2\alpha y$. The minimal embedded resolution of the cusp is already an embedded resolution of $V(fg)$. We have the intersection diagram as in Figure \ref{figure1}. There are no poles of order $2$ and one non-trivial candidate pole, namely $s_0 = -\frac{\nu_3}{N_3} = -\frac{4}{3}$. We calculate
	$$\alpha_0 = \nu_0-\frac{4}{3} \cdot N_0 = -\frac{1}{3}, \quad \alpha_1 = \nu_1-\frac{4}{3}\cdot N_1 = \frac{1}{3}, \quad \alpha_2 = \nu_2 - \frac{4}{3}\cdot N_2 = 1,$$
	which implies by Proposition \ref{contribution} that $s_0= -\frac{4}{3}$ is no pole. (This does not mean there are no poles, there is still the trivial candidate pole $s_0 = -\frac{\nu_0}{N_0} = -1$ arising from $E_0$. In fact, here 
one computes easily that we have 	$\Zetaw{f}{s}{o}{\omega}  = \frac{1}{2(s+1)}$.)
The only pole $s_0 = -1$ induces of course the monodromy eigenvalue $1$.
	
	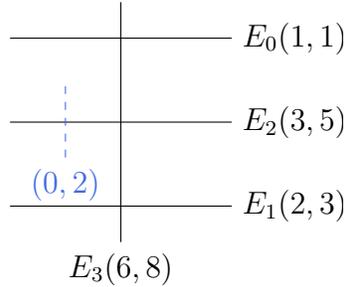
\begin{figure}[h]
		\centering
		\begin{tikzpicture}
			\begin{axis}[
				width=20cm,
				height=6.37cm,
				xmin=-1, xmax=11.5,
				ymin=-1.5, ymax=1.5,
				samples=100,
				hide axis,
				]
				\addplot [black] table {
					8.5 0
					10.5 0
				} node[right] {$E_2 (3,5)$};
				\addplot [black] table {
					8.5 0.7
					10.5 0.7
				} node[right] {$E_0 (1,1)$};
				\addplot [black] table {
					8.5 -0.7
					10.5 -0.7
				} node[right] {$E_1 (2,3)$};
				\addplot [black] table {
					9.5 1
					9.5 -1
				} node[below] {$E_3 (6,8)$};
				\addplot [dashed,hot] table {
					9 0.3
					9 -0.3
				} node[below,hot] {$(0,2)$};
			\end{axis}
		\end{tikzpicture}
		\caption{Intersection diagram for $x^3-y^2$}
		\label{figure1}
	\end{figure}
\end{example}

\begin{example}\label{example3}
	Consider a second example $f = x^5-y^3$; now $\omega = df \wedge dl = (5\beta x^4+3\alpha y^2)dx \wedge dy$. Let $g = 5\beta x^4+3\alpha y^2$. We proceed as in the previous example and obtain the intersection diagram in Figure \ref{figure10} for an embedded resolution of $V(fg)$. Again, there are no poles of order $2$ and two non-trivial candidate poles of order $1$, namely $s_0=-\frac{\nu_2}{N_2} = -\frac{7}{5}$ and $s_0=-\frac{\nu_4}{N_4} = -\frac{6}{5}$. First,
	$$\alpha' = 2-\frac{7}{5} \cdot 0 = 2, \quad \alpha_4 = \nu_4-\frac{7}{5}\cdot N_4 = -3,$$
	such that the residue of $s_0=-\frac{7}{5}$ equals
	$\frac{1}{N_2}\left(\frac{2}{\alpha'}+\frac{1}{\alpha_4}-1\right) = -\frac{1}{15},$ and hence $s_0= -\frac{7}{5}$ is a pole. Second, compute
	$$\alpha_0 = \nu_0-\frac{6}{5} \cdot N_0 = -\frac{1}{5}, \quad \alpha_2 = \nu_2-\frac{6}{5} \cdot N_2 = 1, \quad \alpha_3 = \nu_3-\frac{8}{7}\cdot N_3= \frac{1}{5},$$
	which implies by Proposition \ref{contribution} that  $s_0=-\frac{6}{5}$ is not a pole. Now, the pole $s_0=-\frac{7}{5}$ does not induce a monodromy eigenvalue. Indeed, calculating the monodromy eigenvalues using A'Campo's formula (Theorem \ref{A campo}), we get
	\begin{align*}
		\prod_{j =1}^4 (t^{N_j}-1)^{-\chi(E_j^\circ)} & = \frac{(t^{15}-1)}{(t^3-1)(t^5-1)} 
 =\frac{\Phi_{15}}{\Phi_1}.
	\end{align*}
	This example provides a counterexample to a possible adaptation of the monodromy conjecture.
	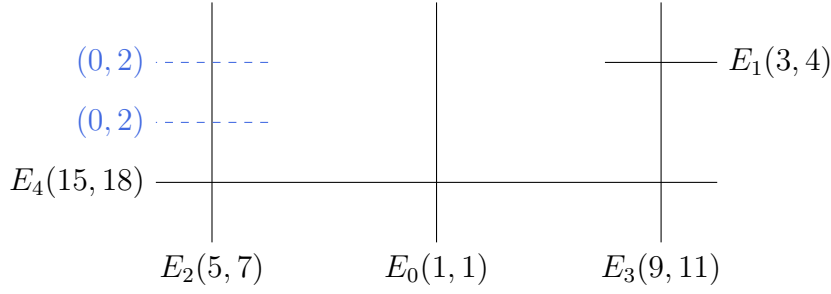
\begin{figure}[h]
		\centering
		\begin{tikzpicture}
			\begin{axis}[domain=-1:1, restrict y to domain=-1:1,
				width=18cm,
				height=6.37cm,
				xmin=-1, xmax=10,
				ymin=-1.5, ymax=1.5,
				samples=100,
				hide axis,
				]
				\addplot [black] table {
					8.5 -0.5
					3.5 -0.5
				} node[left] {$E_4 (15,18)$};
				\addplot [black] table {
					4 1
					4 -1
				} node[below] {$E_2 (5,7)$};
				\addplot [black] table {
					6 1
					6 -1
				} node[below] {$E_0 (1,1)$};
				\addplot [black] table {
					8 1
					8 -1
				} node[below] {$E_3 (9,11)$};
				\addplot [black] table {
					7.5 0.5
					8.5 0.5
				} node[right] {$E_1 (3,4)$};
				\addplot [dashed,hot] table {
					4.5 0.5
					3.5 0.5
				} node[left,hot] {$(0,2)$};
				\addplot [dashed,hot] table {
					4.5 0
					3.5 0
				} node[left,hot] {$(0,2)$};
			\end{axis}
		\end{tikzpicture}
		\caption{Intersection diagram for $x^5-y^3$}
		\label{figure10}
	\end{figure}
\end{example}

We observe that in each example one $\alpha$-value equals $1$, which explains the cancellation of that pole (see Proposition \ref{contribution}). Normally, when using the standard $2$-form $\omega =dx \wedge dy$, we know that every candidate pole arising from an exceptional component intersecting at least three other components is an actual pole, see Theorem \ref{all poles}. This is not the case in the previous two examples. 
In the next subsection, we investigate this phenomenon more generally.

\medskip

\subsection{} First, we introduce the following terminology concerning the \emph{dual graph} associated to an embedded resolution. Its vertices are the components $E_j$, $j \in T$, where the exceptional components are represented by a dot and the components of the strict transform are represented by a circle. Two vertices are connected by an edge if the corresponding components intersect. A \emph{rupture vertex} is an exceptional component intersecting at least three times other components. A \emph{dead branch} is a chain of exceptional components, which starts at a rupture vertex and ends in an exceptional component, that is not created by the first blow-up. We can represent such a dead branch in the dual graph as in Figure \ref{figure31}. It follows from the construction of the resolution graph that to every rupture vertex at most one dead branch is attached. We also add the components of the strict transform of the polar curve, represented by a blue circle, which is connected to its intersecting components by a dotted blue line.

\begin{figure}[H]
	\centering
	\begin{tikzpicture}
		\begin{axis}[domain=-1:1,
			width=15.5cm,
			height=8cm,
			xmin=0.5, xmax=9.5,
			ymin=-2.5, ymax=1,
			samples=100,
			hide axis,
			]
			\addplot [only marks] table {
				5 0
			};
			\addplot [only marks] table {
				6.4 0
			} node[below] {$E_2$};
			\addplot [black] table {
				4.2 0.5
			} node {$\hdots$};
			\addplot [black] table {
				4.2 -0.5
			} node {$\hdots$};
			\addplot [black] table {
				4.7 0
			} node {$\hdots$};
			\addplot [black] table {
				5.7 0
				6.4 0
			};
			\addplot [black] table {
				7 0
			} node {$\hdots$};
			\addplot [black] table {
				6.4 0
				6.75 0
			};
			\addplot [black] table {
				7.25 0
				7.6 0
			};
			\addplot [only marks] table {
				7.6 0
			} node[below] {$E_k$};
			\addplot [only marks] table {
				5.7 0
			} node[below] {$E_1$};
			\addplot [black] table {
				5 0
				5.7 0
			};
			\addplot [black] table {
				4.5 0.5
				5 0
			};
			\addplot [black] table {
				4.5 -0.5	
				5 0
			};
		\end{axis}
	\end{tikzpicture}
	\caption{A dead branch}
	\label{figure31}
\end{figure}
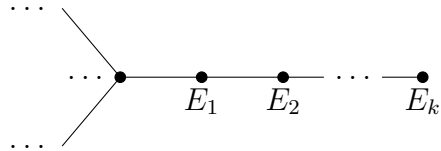

We now look for a possible general phenomenon explaining the unexpected cancellation of candidate poles in Examples \ref{ordinary cusp} and \ref{example3}. That is, given a rupture vertex and an attached dead branch as in Figure \ref{figure31}, under what condition(s) is $\alpha_1:=\nu_1-\frac{\nu}{N} N_1= 1$?
This will of course depend on
where the strict transform of the polar curve intersects the dead branch.
 In \cite{trangmichelweber}, the authors explore where the strict transform of the polar curve intersects the dual graph.

\begin{theorem}[{\cite[Théorème 2.1(4)]{trangmichelweber}}] \label{polar curve thm}
	In the dual graph of the minimal embedded resolution of $f$, the strict transform of the polar curve intersects every dead branch.
\end{theorem}

The precise components in the dead branch that intersect with the strict transform of the polar curve, and conditions when these intersections are tranversal, are very hard to determine. However, in  \cite[\S6]{trangmichelweber}, the authors consider some important families of curves, where this can be done. First, we introduce some necessary results from \cite{le}, and also \cite{teissier}.

\begin{theorem}\label{W_f}
	\begin{enumerate}[wide=0pt]
		\item Let $g \in \C[x,y] \setminus \C$ be a reduced polynomial. Then there exists a dense open set $U_g \subset \Pm^1$ such that, for every linear form $l = \alpha x + \beta y$, $[\alpha:\beta] \in U_g$, the singularity of the product $g \cdot g_l$ has the same topology, where $g_l$ is the polar curve of $g$ with respect to $l$. We call this topological type the polar pair of $g$.
		
		\item Let $f \in \C[x,y] \setminus \C$ be a reduced polynomial and let $\sigma(f)$ be the $\mu$-constant stratum of $f$. There exists a Zariski dense open subset $W_f \subset \sigma(f)$ such that the polar pair is constant for all curves belonging to $W_f$.
	\end{enumerate}
\end{theorem}

Note that it is not necessarily true that $f$ belongs to $W_f$.
We have the following result from \cite[Theorem 1]{casas} and \cite{casas2}, where we also use the appendix of \cite{trangmichelweber}.

\begin{theorem} \label{casas thm}
	\begin{enumerate}[wide=0pt]
		\item Let $f=y^b-x^a$, with $a$ and $b$ coprime and $a>b$. Consider the continued fraction expansion
		$$\frac{a}{b}=q_0+\frac{1}{q_1+\hdots+\frac{1}{q_{w-1}+\frac{1}{q_w}}},$$
		where $q_0 \geq 0$ and $q_j \geq 1$ for $j=1,\hdots,w$. Assume that the length $w$ is even. Set $s_k := \sum_{i=0}^kq_i$ for $k = 0,\hdots,w$. Then, for every curve in $W_f$, the dual graph of the minimal embedded resolution of the curve is already an embedded resolution for its polar curve, and the intersections are given by Figure \ref{figure34}. Here the numerical data of every component of the polar curve are $(0,2)$.
		\begin{figure}[H]
			\centering
			\begin{tikzpicture}
				\begin{axis}[domain=-1:1,
					width=18cm,
					height=5cm,
					xmin=0, xmax=11,
					ymin=-1.5, ymax=1.5,
					samples=100,
					hide axis,
					]
					\addplot [only marks] table {
						0.5 0
					} node[above] {$s_w$};
					\addplot [only marks, mark=o, mark options={black}] table {
						0 0
					};
					\addplot [black] table {
						0.05 0
						2.75 0
					};
					
					\addplot [only marks] table {
						1.5 0
					} node[below] {$s_{w-1}$};
					\addplot [only marks] table {
						2.5 0
					} node[above] {$s_{w-1}-1$};
					\addplot [black] table {
						3 0
					} node {$\hdots$};
					\addplot [black] table {
						3.25 0
						5.75 0
					};
					
					\addplot [dashed,hot] table {
						1.03 0.7
						1.5 0
					};
					\addplot [dashed,hot] table {
						1.97 0.7
						1.5 0
					};
					\addplot [only marks, mark=o, mark options={hot}] table {
						1 0.75
					};
					\addplot [only marks, mark=o, mark options={hot}] table {
						2 0.75
					};
					\addplot [black] table {
						1.5 0.75
					} node[hot] {$\hdots$};
					\addplot [black] table {
						1.5 1
					} node[hot] {$q_w$ times};
					
					\addplot [only marks] table {
						3.5 0
					} node[above] {$s_{2j}+1$};
					\addplot [only marks] table {
						4.5 0
					} node[below] {$s_{2j-1}$};
					\addplot [only marks] table {
						5.5 0
					} node[above] {$s_{2j-1}-1$};
					\addplot [black] table {
						6 0
					} node {$\hdots$};
					\addplot [black] table {
						6.25 0
						8.75 0
					};
					
					\addplot [dashed,hot] table {
						4.03 0.7
						4.5 0
					};
					\addplot [dashed,hot] table {
						4.97 0.7
						4.5 0
					};
					\addplot [only marks, mark=o, mark options={hot}] table {
						4 0.75
					};
					\addplot [only marks, mark=o, mark options={hot}] table {
						5 0.75
					};
					\addplot [black] table {
						4.5 0.75
					} node[hot] {$\hdots$};
					\addplot [black] table {
						4.5 1
					} node[hot] {$q_{2j}$ times};
					
					\addplot [only marks] table {
						6.5 0
					} node[above] {$s_{2}+1$};
					\addplot [only marks] table {
						7.5 0
					} node[below] {$s_{1}$};
					\addplot [only marks] table {
						8.5 0
					} node[above] {$s_{1}-1$};
					\addplot [black] table {
						9 0
					} node {$\hdots$};
					\addplot [black] table {
						9.25 0
						10.5 0
					};
					
					\addplot [dashed,hot] table {
						7.03 0.7
						7.5 0
					};
					\addplot [dashed,hot] table {
						7.97 0.7
						7.5 0
					};
					\addplot [only marks, mark=o, mark options={hot}] table {
						7 0.75
					};
					\addplot [only marks, mark=o, mark options={hot}] table {
						8 0.75
					};
					\addplot [black] table {
						7.5 0.75
					} node[hot] {$\hdots$};
					\addplot [black] table {
						7.5 1
					} node[hot] {$q_{2}$ times};
					
					\addplot [only marks] table {
						9.5 0
					} node[below] {$s_{0}+2$};
					\addplot [only marks] table {
						10.5 0
					} node[above] {$s_{0}+1$};
					
					\addplot [only marks] table {
						1.5 -1
					} node[below] {$s_w-1$};
					\addplot [black] table {
						0.5 0
						1.5 -1
					};
					\addplot [black] table {
						1.75 -1
						1.5 -1
					};
					\addplot [black] table {
						2 -1
					} node {$\hdots$};
					\addplot [black] table {
						2.25 -1
						4.75 -1
					};
					\addplot [only marks] table {
						2.5 -1
					} node[below] {$s_1+1$};
					\addplot [only marks] table {
						3.5 -1
					} node[below] {$s_0$};
					\addplot [only marks] table {
						4.5 -1
					} node[below] {$s_0-1$};
					\addplot [black] table {
						5 -1
					} node {$\hdots$};
					\addplot [black] table {
						5.25 -1
						6.5 -1
					};
					\addplot [only marks] table {
						5.5 -1
					} node[below] {$2$};
					\addplot [only marks] table {
						6.5 -1
					} node[below] {$1$};
				\end{axis}
			\end{tikzpicture}
			\caption{Dual graph}
			\label{figure34}
		\end{figure}
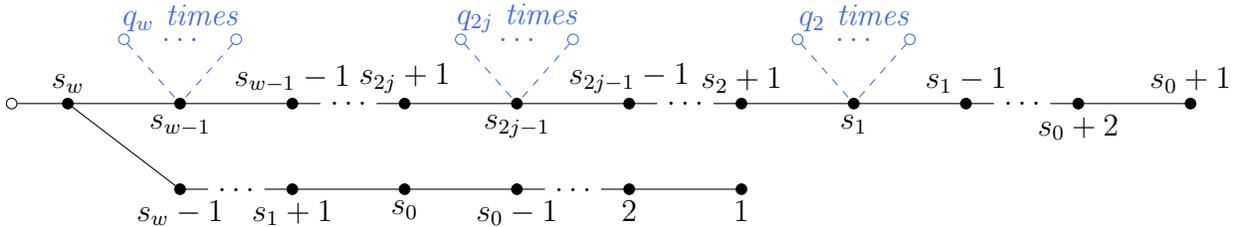
		
		\item More generally, let $f$ be unibranch with 
Zariski pairs $\{(a_1,b_1),\hdots, (a_g,b_g)\}$ ($\gcd(a_i,b_i)=1$). Assume that the length of the continued fractions of each $\frac{a_i}{b_i}$ is even. Let $B(a_i,b_i)$ be the dual graph of the minimal embedded resolution of $y^{b_i}-x^{a_i}$. Then the dual graph of the minimal embedded resolution of a curve in $W_f$ is formed by first replacing the component of the strict transform of $B(a_1,b_1)$ with the first exceptional component of $B(a_2,b_2)$ and adding the rest of $B(a_2,b_2)$, then repeating the same process with $B(a_2,b_2)$ and $B(a_3,b_3)$, and continuing this way until adding $B(a_g,b_g)$. This means that the dead branch attached to a rupture vertex is the same as in $B(a_i,b_i)$.

Again, this resolution is already an embedded resolution for the polar curve, and its strict transform intersects each dead branch as in Figure \ref{figure34}.
	\end{enumerate}
\end{theorem}

Using Theorem \ref{casas thm}, we deduce the following result.

\begin{proposition}\label{prop1}
	\begin{enumerate}[wide=0pt]
		\item Let $f=y^b-x^a$ ($\gcd(a,b)=1$). Assume that the length of the continued fraction of $\frac{a}{b}$ is even. Let $E_0$ be the rupture vertex in the dual graph associated to the minimal embedded resolution of a curve in $W_f$, and let the dead branch attached to $E_0$ start at $E_1$. Then
		\begin{equation}\label{alpha 1}
			\alpha_1 = \nu_1-\frac{\nu_0}{N_0} \cdot N_1=1.
		\end{equation}
		
		\item More generally, let $f$ have Zariski pairs $\{(a_1,b_1),\hdots, (a_g,b_g)\}$ ($\gcd(a_i,b_i)=1)$).
		Assume that the lengths of the continued fractions of all $\frac{a_i}{b_i}$ are even. Let $E_0$ be a rupture vertex in the dual graph associated to the minimal embedded resolution of a curve in $W_f$, and let the dead branch attached to $E_0$ start at $E_1$, then
		\begin{equation}\label{alpha 2}
			\alpha_1 = \nu_1-\frac{\nu_0}{N_0} \cdot N_1=1.
		\end{equation}
	\end{enumerate}
\end{proposition}

\begin{proof}
	\begin{enumerate}[wide]
		\item Using Theorem \ref{casas thm}, we have the dual graph as in Figure \ref{figure34} of the minimal embedded resolution of the curve (which is already an embedded resolution for its polar curve). From this graph we can deduce the self-intersection numbers of the exceptional components, see for example \cite[Section 8.5]{wall2}. Set $\kappa_j = -E_j^2$ for $j=1,\hdots,s_w-1$. Then
\begin{equation}\label{kappavalues}
\left\{ \begin{array}{ll}
			\kappa_j = \kappa_{s_k+i} = 2 &\text{ for } j=1,\hdots,q_0-1, \ k=0,\hdots,w-1 \text{ and } i=1,\hdots,q_{k+1}-1,\\
			\kappa_{s_k} = q_{k+1}+2 & \text{ for } k=0,\hdots,w-2,\\
			\kappa_{s_{w-1}} = q_{w}+1.  
		\end{array} \right.
\end{equation}
		We rename the components in the dead branch to simplify notation, see Figure \ref{figure35}.
		\begin{figure}[H]
			\centering
			\begin{tikzpicture}
				\begin{axis}[domain=-1:1,
					width=18cm,
					height=5cm,
					xmin=0, xmax=11,
					ymin=-1.5, ymax=1.5,
					samples=100,
					hide axis,
					]
					\addplot [only marks] table {
						0.5 0
					} node[above] {$E_0$};
					\addplot [only marks, mark=o, mark options={black}] table {
						0 0
					};
					\addplot [black] table {
						0.05 0
						2.75 0
					};
					\addplot [black] table {
						0.5 0
						0.5 -0.5
					};
					\addplot [black] table {
						0.5 -0.7
					} node {$\vdots$};
					
					\addplot [only marks] table {
						1.5 0
					} node[below] {$E_1$};
					\addplot [only marks] table {
						2.5 0
					} node[above] {$E_2$};
					\addplot [black] table {
						3 0
					} node {$\hdots$};
					\addplot [black] table {
						3.25 0
						5.75 0
					};
					
					\addplot [dashed,hot] table {
						1.03 0.7
						1.5 0
					};
					\addplot [dashed,hot] table {
						1.97 0.7
						1.5 0
					};
					\addplot [only marks, mark=o, mark options={hot}] table {
						1 0.75
					};
					\addplot [only marks, mark=o, mark options={hot}] table {
						2 0.75
					};
					\addplot [black] table {
						1.5 0.75
					} node[hot] {$\hdots$};
					\addplot [black] table {
						1.5 1
					} node[hot] {$\kappa_1-1$ times};
					
					\addplot [only marks] table {
						3.5 0
					} node[above] {$E_{j-1}$};
					\addplot [only marks] table {
						4.5 0
					} node[below] {$E_j$};
					\addplot [only marks] table {
						5.5 0
					} node[above] {$E_{j+1}$};
					\addplot [black] table {
						6 0
					} node {$\hdots$};
					\addplot [black] table {
						6.25 0
						7.5 0
					};
					
					\addplot [dashed,hot] table {
						4.03 0.7
						4.5 0
					};
					\addplot [dashed,hot] table {
						4.97 0.7
						4.5 0
					};
					\addplot [only marks, mark=o, mark options={hot}] table {
						4 0.75
					};
					\addplot [only marks, mark=o, mark options={hot}] table {
						5 0.75
					};
					\addplot [black] table {
						4.5 0.75
					} node[hot] {$\hdots$};
					\addplot [black] table {
						4.5 1
					} node[hot] {$\kappa_j-2$ times};
					
					\addplot [only marks] table {
						6.5 0
					} node[above] {$E_{r-1}$};
					\addplot [only marks] table {
						7.5 0
					} node[above] {$E_r$};
				\end{axis}
			\end{tikzpicture}
			\caption{Dead branch in the dual graph}
			\label{figure35}
		\end{figure}
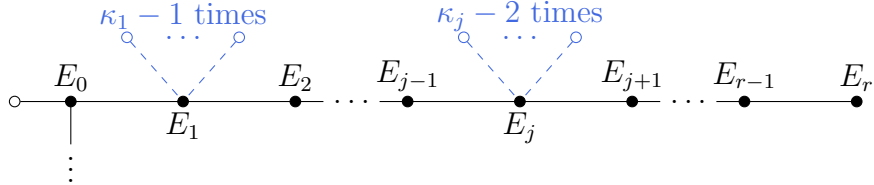
		We assume that $r >1$; if $r=1$ one can easily verify the result using Lemma \ref{num lemma}. Also from Lemma \ref{num lemma}, we have that
		$$\left\{ \begin{array}{ll}
			\kappa_{j}N_{j} - N_{j-1}-N_{j+1}= 0,\\
			\kappa_{j}\nu_{j} - \nu_{j-1}-\nu_{j+1}= \kappa_j-2,
		\end{array} \right.$$
		for $j=2,\hdots,r-1$.
In the second equation, note that, by Figure \ref{figure34} and the equalities (\ref{kappavalues}), we have that a component $E_j$ intersects the strict transform of the polar curve if and only if  $\kappa_j-2 > 0$.
		For the ends of the dead branch, $E_1$ and $E_r$, we have the following (also using Lemma \ref{num lemma}):
		$$\begin{array}{ccc}
					\left\{
					\begin{array}{l}
							\kappa_1N_{1}-N_2 = N_0,\\
							\kappa_1\nu_{1} -\nu_2=\nu_0+\kappa_1-1,
						\end{array}
					\right.
					&
					\quad \text{and} \quad
					&
					\left\{
					\begin{array}{l}
							\kappa_rN_{r} -N_{r-1}=0,\\
							\kappa_r\nu_{r} -\nu_{r-1} = \kappa_r-1.
						\end{array}
					\right.
				\end{array}$$
		Hence, we have the following equations using the intersection matrix $A= (-E_i\cdot E_j)_{1\leq i,j\leq r}$:
		$$\begin{array}{ccc}
			A\cdot  \begin{pmatrix}
			N_1\\
			N_2\\
			\vdots\\
			N_{r-1}\\
			N_r
			\end{pmatrix} = \begin{pmatrix}
			N_0\\
			0\\
			\vdots\\
			0\\
			0
			\end{pmatrix}
			&
			\quad \text{and} \quad
			&
			A\cdot \begin{pmatrix}
				\nu_1\\
				\nu_2\\
				\vdots\\
				\nu_{r-1}\\
				\nu_r
			\end{pmatrix} = \begin{pmatrix}
				\nu_0+\kappa_1-1\\
				\kappa_2-2\\
				\vdots\\
				\kappa_{r-1}-2\\
				\kappa_r-1
			\end{pmatrix}.
		\end{array}$$
		By inverting $A$,  we compute that
		$$\left\{ \begin{array}{ll}
			N_{1} = \frac{\Delta_{2,r}}{\Delta_{1,r}}N_0,\\
			\nu_{1} = \frac{\Delta_{2,r}}{\Delta_{1,r}}\nu_0+\frac{1}{\Delta_{1,r}}\big(\Delta_{2,r}(\kappa_1-1)+\Delta_{3,r}(\kappa_2-2)+\hdots+\Delta_{r,r}(\kappa_{r-1}-2)+\kappa_r-1\big),
		\end{array} \right.$$
		where $\Delta_{k,l} := \det((- E_i\cdot E_j)_{k\leq i,j \leq l})$, for $k \leq l$  (so $E_k,E_{k+1},\hdots E_l$ form a subchain in the dead branch). Then
		$$\alpha_1 = \nu_1-\frac{\nu_0}{N_0} \cdot N_1=\frac{1}{\Delta_{1,r}}(\Delta_{2,r}(\kappa_1-1)+\Delta_{3,r}(\kappa_2-2)+\hdots+\Delta_{r,r}(\kappa_{r-1}-2)+\kappa_r-1).$$
		Thus, Equality \eqref{alpha 1} is equivalent to
		\begin{equation}\label{matrix det}
			\Delta_{1,r} = \Delta_{2,r}(\kappa_1-1)+\Delta_{3,r}(\kappa_2-2)+\hdots+\Delta_{r,r}(\kappa_{r-1}-2)+\kappa_r-1.
		\end{equation}
		We prove this by induction on the size of $A$. The case $r=2$ is straightforward. We assume that the result holds for the matrix  $ (-E_i\cdot E_j)_{2\leq i,j\leq r}$.
Expanding the determinant of $A$ via the first row, we obtain (the well-known equality)
		$$\Delta_{1,r} = \Delta_{2,r}\kappa_1-\Delta_{3,r} =\Delta_{2,r}(\kappa_1-1)  + \Delta_{2,r}-\Delta_{3,r} .$$
		Applying the induction hypothesis to the second  $\Delta_{2,r}$ above, we obtain
		\begin{align*}
			\Delta_{1,r} & =\Delta_{2,r} (\kappa_1-1)
+ \big[ \Delta_{3,r}(\kappa_2-1)+\Delta_{4,r}(\kappa_3-2)+\hdots+\Delta_{r,r}(\kappa_{r-1}-2)+\kappa_r-1 \big]  -\Delta_{3,r}\\
&=\Delta_{2,r}(\kappa_1-1)+\Delta_{3,r}(\kappa_2-2)+\hdots+\Delta_{r,r}(\kappa_{r-1}-2)+\kappa_r-1.
		\end{align*}
				
		\item Using Theorem \ref{casas thm}, we know that the dead branch attached to a rupture vertex is the same as in $B(a_i,b_i)$, as well as the intersections by components of the strict transform of the polar curve. This also means that the self-intersection numbers remain the same (except for possibly the rupture vertex). Then precisely the same argument as in (1) implies (\ref{alpha 2}).
	\end{enumerate}
\end{proof}

So, Proposition \ref{prop1} \lq explains\rq\ in particular Example \ref{example3}. Is a similar result true when the length of some continued fraction is odd?  Or when $f$ is an analytically reducible polynomial?
We calculate an example for each case. We present the dual graphs and the relevant numerical data, i.e., every $\alpha$-value corresponding to a rupture vertex with an attached dead branch.

\begin{figure}[h]
	\centering
	
	\begin{subfigure}{\textwidth}
		\centering
		\begin{tikzpicture}
			\begin{axis}[domain=-1:1,
				width=18cm,
				height=10cm,
				xmin=-1.5, xmax=7.5,
				ymin=-8, ymax=3,
				samples=100,
				hide axis,
				]
				\addplot [only marks] table {
					0 -1
				} node[right,xshift=0.5cm] {$E_{11} (530,540)$};
				\addplot [black] table {
					0 -1
					-0.97 -1
				};
				\addplot [black] table {
					0 -1
					1 -2.5
				};
				\addplot [only marks, mark=o, mark options={black}] table {
					-1 -1
				} node[above] {$E_0 (1,1)$};
				\addplot [only marks] table {
					1 -2.5
				} node[below] {$E_{10} (370,377)$};
				\addplot [black] table {
					1 -2.5
					2 -2.5
				};
				\addplot [only marks] table {
					2 -2.5
				} node[above] {$E_{9} (210,214)$};
				\addplot [black] table {
					2 -2.5
					3 -2.5
				};
				\addplot [only marks] table {
					3 -2.5
				} node[below] {$E_{5} (50,51)$};
				\addplot [black] table {
					3 -2.5
					4 -2.5
				};
				\addplot [only marks] table {
					4 -2.5
				} node[above] {$E_{4} (40,41)$};
				\addplot [black] table {
					4 -2.5
					5 -2.5
				};
				\addplot [only marks] table {
					5 -2.5
				} node[below] {$E_{3} (30,31)$};
				\addplot [black] table {
					5 -2.5
					6 -2.5
				};
				\addplot [only marks] table {
					6 -2.5
				} node[above] {$E_{2} (20,21)$};
				\addplot [black] table {
					6 -2.5
					7 -2.5
				};
				\addplot [only marks] table {
					7 -2.5
				} node[below] {$E_{1} (10,11)$};
				\addplot [black] table {
					0 -1
					1 0.5
				};
				\addplot [only marks] table {
					1 0.5
				} node[above] {$E_{8} (159,163)$};
				\addplot [black] table {
					1 0.5
					2 0.5
				};
				\addplot [only marks] table {
					2 0.5
				} node[below] {$E_{7} (106,112)$};
				\addplot [black] table {
					2 0.5
					3 0.5
				};
				\addplot [only marks] table {
					3 0.5
				} node[above] {$E_{12} (159,173)$};
				\addplot [black] table {
					3 0.5
					4 0.5
				};
				\addplot [only marks] table {
					4 0.5
				} node[below] {$E_{13} (212,234)$};
				\addplot [black] table {
					4 0.5
					5 0.5
				};
				\addplot [only marks] table {
					5 0.5
				} node[above] {$E_{15} (477,529)$};
				\addplot [dashed,hot] table {
					5 0.5
					5 -0.45
				};
				\addplot [only marks, mark=o, mark options={hot}] table {
					5 -0.5
				} node[below,hot] {$(0,2)$};
				\addplot [black] table {
					5 0.5
					6 0.5
				};
				\addplot [only marks] table {
					6 0.5
				} node[below] {$E_{14} (265,294)$};
				\addplot [black] table {
					6 0.5
					7 0.5
				};
				\addplot [only marks] table {
					7 0.5
				} node[above] {$E_{6} (53,59)$};
			\end{axis}
		\end{tikzpicture}
		\caption{Dual graph for $f = y^{10}-x^{53}$}
		\label{figureB4}
	\end{subfigure}
	
	\medskip
	\begin{subfigure}{\textwidth}
		\centering
		\begin{tikzpicture}
			\begin{axis}[domain=-1:1,
				width=15cm,
				height=8cm,
				xmin=-0.25, xmax=9.75,
				ymin=-2, ymax=2,
				samples=100,
				hide axis,
				]
				\addplot [only marks] table {
					2.5 0
				} node[left] {$E_4 (14,16)$};
				\addplot [black] table {
					2.5 0
					4 0
				};
				\addplot [only marks] table {
					4 0
				} node[below] {$E_2 (6,7)$};
				\addplot [black] table {
					4 0
					5.5 0
				};
				\addplot [only marks] table {
					5.5 0
				} node[below] {$E_1 (4,5)$};
				\addplot [black] table {
					5.5 0
					7 0
				};
				\addplot [only marks] table {
					7 0
				} node[right] {$E_6 (10,12)$};
				\addplot [black] table {
					2.5 0
					1.78 0.77
				};
				\addplot [only marks, mark=o, mark options={black}] table {
					1.75 0.8
				} node[above] {$E_0 (1,1)$};
				\addplot [black] table {
					2.5 0
					1.75 -0.8
				};
				\addplot [only marks] table {
					1.75 -0.8
				} node[below] {$E_3 (7,9)$};
				\addplot [dashed,hot] table {
					0.3 -0.8
					1.75 -0.8
				};
				\addplot [only marks, mark=o, mark options={hot}] table {
					0.25 -0.8
				} node[below,hot] {$(0,2)$};
				\addplot [black] table {
					7 0
					7.75 -0.8
				};
				\addplot [only marks] table {
					7.75 -0.8
				} node[below] {$E_5 (5,7)$};
				\addplot [black] table {
					7 0
					7.72 0.77
				};
				\addplot [only marks, mark=o, mark options={black}] table {
					7.75 0.8
				} node[above] {$E_0 (1,1)$};
				\addplot [dashed,hot] table {
					7.75 -0.8
					9.2 -0.8
				};
				\addplot [only marks, mark=o, mark options={hot}] table {
					9.25 -0.8
				} node[below,hot] {$(0,2)$};
				\addplot [dashed,hot] table {
					5.5 0
					5.5 0.75
				};
				\addplot [only marks, mark=o, mark options={hot}] table {
					5.5 0.8
				} node[above,hot] {$(0,2)$};
			\end{axis}
		\end{tikzpicture}
		\caption{Dual graph for $f = x^2y^2-x^5-y^7$}
		\label{figureB5}
	\end{subfigure}
	
	\caption{Dual graphs for additional examples}
	\label{fig:combined}
\end{figure}
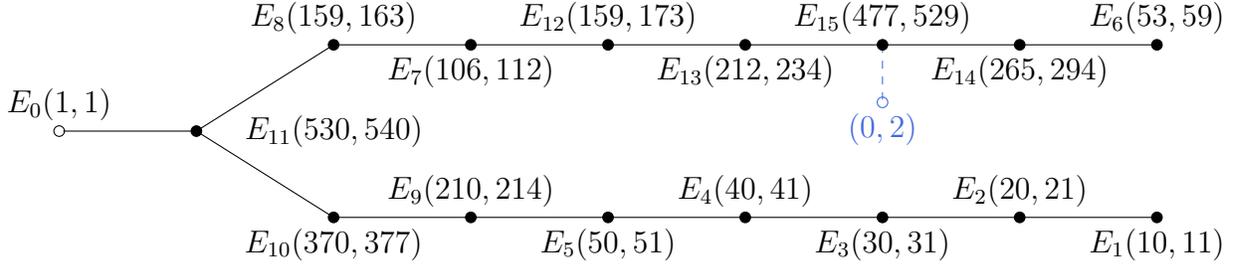
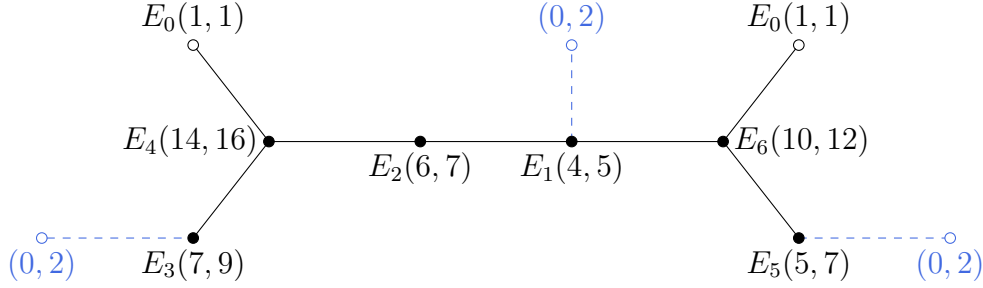

\begin{example}
	Let $f = y^{10}-x^{53}$, so $w=1$ is odd. An embedded resolution of $f$ and its polar curve is given by Figure \ref{figureB4}. Note that the minimal embedded resolution of $f$ is given by eleven blow-ups, but for an embedded resolution of $f$ and its polar curve four more blow-ups are needed, creating $E_{12},E_{13},E_{14}$ and $E_{15}$. We calculate for $E_{11}$ that
	$$\alpha_8 = \nu_8 - \frac{\nu_{11}}{N_{11}} \cdot N_8 = 163-\frac{54}{53} \cdot 159 = 1.$$
\end{example}

\begin{example}
	Let $f = x^2y^2-x^5-y^7$, so $f$ is analytically reducible (see Figure \ref{figureB5}). An embedded resolution of $f$ and its polar curve is given by Figure \ref{figureB5}. We calculate for $E_4$ that
	$$\alpha_3 = \nu_3-\frac{\nu_4}{N_4}\cdot N_3 = 9-\frac{8}{7} \cdot 7 = 1,$$
 and for $E_6$ that
	$$\alpha_5 = \nu_5-\frac{\nu_6}{N_6}\cdot N_5 = 7-\frac{6}{5} \cdot 5 = 1.$$
\end{example}

Motivated by these additional examples, we can wonder whether we have {\em always} that $\alpha_1=1$ in such a setting. 
More precisely, let $E_0$ be a rupture vertex in the dual graph associated to the minimal embedded resolution of $f$, with an attached dead branch consisting of the chain $E_1-E_2-\ldots-E_r$, where $E_1$ is attached to $E_0$.  Is then $\alpha_1=1$?

It turns out that this question is equivalent to a question about polar curves of independent interest. 
Namely, let $m_i$ denote the total order of contact between $E_i$ and the strict transform of the polar curve of $f$, for $i=1,\dots,r$. Now, the second equation of Lemma \ref{num lemma} becomes
$$\kappa_i\nu_i = \sum_{j=1}^k(\nu'_j-1)+2+m_i,$$
for $i=1,\hdots,r$ and $E_i$ intersecting exactly $k$ times other components $E'_1,\hdots,E'_k$. One can deduce from this that the equality $\alpha_1=1$ is equivalent to
	\begin{equation}\label{question}
	\Delta_{1,r} = \Delta_{2,r}m_1+\Delta_{3,r}m_2+\hdots+\Delta_{r,r}m_{r-1}+m_r +1,
	\end{equation}
in exactly the same manner as in Proposition \ref{prop1} (see the equality \eqref{matrix det}).

\begin{openquestion}
	Let $E_0$ be a rupture vertex in the dual graph associated to the minimal embedded resolution of $f$, with an attached dead branch $E_1-E_2-\ldots-E_r$, which starts at $E_1$. Let $m_i$ denote the total order of contact between $E_i$ and the strict transform of the polar curve of $f$, for $i=1,\dots,r$, and put $\Delta_{k,r} := \det((- E_i\cdot E_j)_{k\leq i,j \leq r})$, for $k=1,\dots,r$.  Then the equality \eqref{question} holds.
\end{openquestion}



\begin{thebibliography}{dFKX17}
	
\bibitem[AC75]{acampo} A'Campo, N. La fonction zêta d'une monodromie. Comment. Math. Helv. 50 (1975), 233-248.

\bibitem[Bo18]{borisov} Borisov, L. Class of the affine line is a zero divisor in the Grothendieck ring. J. Algebraic Geom. 27 (2018), 203-209.

\bibitem[Ca83]{casas} Casas, E. On the singularities of polar curves. Manuscripta Math 43 (1983), 167-190.

\bibitem[Ca90]{casas2} Casas, E. Infinitely near imposed singularities and singularities of polar curves. Mathematische Annalen 287 (1990), 429-454.

\bibitem[CNS18]{cns}  Chambert-Loir A., Nicaise J., and Sebag, J. Motivic integration, vol. 325 of Progr. Math. Birkhäuser/Springer, New York, 2018.

\bibitem[Ch24]{cherik} Cherik, Y. Lipschitz geometry of complex surface germs via inner rates of primary ideals. To appear in Ann. Inst. Fourier, preprint (2024): \url{https://arxiv.org/abs/2407.14265}.

\bibitem[DL91]{denefloeser} Denef, J. and Loeser, F. Charactéristiques d'Euler-Poincaré, fonctions zêta locales et modification analytiques. J. Amer. Math. Soc. 5 (1991), 705-720.

\bibitem[DL98]{denefloeser2} Denef, J. and Loeser, F. Motivic Igusa functions. J. Alg. Geom. 7 (1998), 505-537.

	

\bibitem[Le71]{le} Le, D.T. Sur un critère d'équisingularité. CRAS 272 Série A (1971), 138-140.

\bibitem[LMW89]{trangmichelweber} Le, D.T., Michel, F. and Weber, C. Sur le component des polaires associées aux germes de courbes planes. Compositio Mathematica tome 72, n°1 (1989), 87-113.

\bibitem[Lo88]{loeser} Loeser, F. Fonctions d'Igusa $p$-adiques et Polynômes de Bernstein. Amer. J. Math. 110 (1988), 1-21.

\bibitem[Mi68]{milnor} Milnor, J. Singular points of complex hypersurfaces. Annals of Mathematics Studies 61 (1968), iii+122 pp.
	

\bibitem[Ne22]{nemethi} Némethi, A. Normal Surface Singularities, vol. 74 of Ergebnisse der Mathematik und ihrer Grenzgebiete, Folge. 3. Springer, 2022.

\bibitem[NV12]{nemethiveys} Némethi, A. and Veys, W. Generalized monodromy conjecture in dimension two. Geom. Topol.
16 (2012), 155–217.


\bibitem[Po02]{poonen} Poonen, B. The Grothendieck ring of varieties is not a domain. Math. Res. Letters. 9 (2002), 493-497.
	
\bibitem[Ro04]{rodrigues} Rodrigues, B. On the monodromy conjecture for curves on normal surfaces. Math. Proc. Camb. Phil. Soc. 136 (2004), 313-324.

\bibitem[RV03]{bartenveys} Rodrigues, B. and Veys, W. Poles of zeta functions on normal surfaces. Proc. London Math. Soc. 87 (2003), 439-456.

\bibitem[Te77]{teissier} Teissier B. Variétés polaires I. Invariants polaires des singularités d’hypersurfaces. Invent. Math. 40 (1977) 267-292.

\bibitem[Ve95]{veys3} Veys, W. Determination of the poles of the topological zeta function for curves. Manuscripta Math. 87 (1995), 435-448.

\bibitem[Ve99]{veys2} Veys, W. The topological zeta function associated to a function on a normal surface germ. Topology 38 (1999), 439-456.

\bibitem[Ve01]{veys4} Veys W. Zeta Functions and ‘Kontsevich Invariants’ on Singular Varieties. Canadian Journal of Mathematics 53, 4 (2001), 834-865. doi:10.4153/CJM-2001-034-1

\bibitem[Ve25]{veys} Veys, W. Introduction to the Monodromy Conjecture. Springer Nature Switzerland, Cham, 2025, pp. 721-765.

\bibitem[Wa00]{wall} Wall, C.T.C. Quadratic forms and normal surface singularities. Contemporary Mathematics 272 (A.M.S., 2000), 293-311.

\bibitem[Wa04]{wall2} Wall, C.T.C. Singular Point of Plane Curves. London Mathematical Society Student Texts. Cambridge University Press, 2004.

\end{thebibliography}
\end{document}